\documentclass[12pt]{article}

\usepackage[utf8]{inputenc}
\usepackage[T1]{fontenc}
\usepackage{titlesec}
\usepackage[margin = 70pt, top = 35mm, bottom = 35mm, headheight=20pt]{geometry} 
\usepackage{amsmath}
\usepackage{amsthm}
\usepackage{hyperref}
\usepackage{amssymb}
\usepackage{amsrefs}
\usepackage{fancyhdr}
\usepackage{parskip}[=v1]
\usepackage{graphicx}
\usepackage{tikz-cd} 
\usepackage{etoolbox}

\usetikzlibrary{cd}

\footskip =  \dimexpr\headsep+\ht\strutbox\relax
\fancypagestyle{toc}{%
\fancyhf{}%
\fancyfoot[L]{\footnotesize A review on coisotropic reduction in Symplectic, Cosymplectic, Contact and Co-contact systems}%
\fancyfoot[R]{\thepage}

}

\usepackage[english]{babel}
\newtheorem{prop}{Proposition}[section] 
\newtheorem{theorem}{Theorem}[section] 
\newtheorem{Def}{Definition}[section]
\newtheorem{lemma}{Lemma}[section]
\newtheorem{corollary}{Corollary}[section]
\newtheorem{remark}{Remark}[section]

\renewcommand{\dim}{\operatorname{dim}} 
\renewcommand{\ker}{\operatorname{Ker}}
\newcommand{\im}{\operatorname{Im}}
\newcommand{\grad}{\operatorname{grad}}
\newcommand{\supp}{\operatorname{supp}}
\renewcommand{\d}{\operatorname{d}}
\newcommand{\sectiontitle}{section}
\newcommand{\setsectiontitle}[1]{\renewcommand{\sectiontitle}{\footnotesize\textit{#1}}}
\newcommand{\partder}[2]{\frac{\partial #1}{\partial #2}}
\newcommand{\totder}[2]{\frac{d #1}{d #2}}
\renewcommand{\supp}{\operatorname{supp}}

\title{\myfont A review on coisotropic reduction in Symplectic, Cosympletic, Contact and Co-contact Hamiltonian systems} 
\author{Rubén Izquierdo-López, Manuel de León}
\date{}

\titleformat{\section}{ \large \normalfont \bfseries  \color{blue}}{\textcolor{blue}{\thesection.}}{0.5em}{}
\titleformat{\subsection}{ \normalfont \bfseries  \color{blue}}{\textcolor{blue}{\thesubsection.}}{0.5em}{}
\titleformat{\subsubsection}{ \normalfont \bfseries  \color{blue}}{\textcolor{blue}{\thesubsubsection.}}{0.5em}{}

\begin{document} 
\font \myfont=cmr12 at 25pt

\thispagestyle{toc}
\begin{center}
{\myfont A review on coisotropic reduction in Symplectic, Cosymplectic, Contact and Co-contact Hamiltonian systems}
\end{center}
\hfill

\begin{center}

Manuel de Le\'on\footnote{E-mail: \href{mailto:mdeleon@icmat.es}{mdeleon@icmat.es}}
\\ Instituto de Ciencias Matem\'aticas, Campus Cantoblanco \\
 Consejo Superior de Investigaciones Cient\'ificas
 \\
C/ Nicol\'as Cabrera, 13--15, 28049, Madrid, Spain
\\
and
\\
Real Academia Espa{\~n}ola de las Ciencias.
\\
C/ Valverde, 22, 28004 Madrid, Spain.

\bigskip

Rubén Izquierdo-López\footnote{E-mail: \href{mailto:rubizqui@ucm.es}{rubizqui@ucm.es}}
\\ Instituto de Ciencias Matem\'aticas, Campus Cantoblanco \\
 Consejo Superior de Investigaciones Cient\'ificas
 \\
C/ Nicol\'as Cabrera, 13--15, 28049, Madrid, Spain
\end{center}
\hfill


\begin{abstract} 
In this paper we study coisotropic reduction in different types of dynamics according to the geometry of the corresponding phase space. The relevance of coisotropic reduction is motivated by the fact that these dynamics can always be interpreted as Lagrangian or Legendrian submanifolds. Furthermore, Lagrangian or Legendrian submanifolds can be reduced by a coisotropic one.

\smallskip

\smallskip


 \end{abstract}

 \tableofcontents
 
\section{Introduction}
\setsectiontitle{INTRODUCTION}

The introduction of symplectic geometry in the study of Hamiltonian systems was a tremendous breakthrough, both in quantitative and qualitative aspects. For example, we have the results in the reduction of the original Hamiltonian system when in the presence of symmetries, or the so-called coisotropic reduction \cites{abraham2008foundations,arnold,libermann2012symplectic,de2011methods}. Another relevant example, in the quantitative aspects, is the development of geometric integrators that respect geometric aspects and prove to be more efficient than the traditional ones (see for instance \cites{sanzserna, marsdenbook}). It has also had a major influence on the study of completely integrable systems and Hamilton-Jacobi theory \cites{arnold,abraham2008foundations,catalanes,vaquero}. In addition, the so-called geometric quantization relies on symplectic geometry \cites{kostant,SouriauDynamical}.

Regarding the reduction in the presence of symmetries, the most relevant result is the so-called Marsden-Weinstein symplectic reduction theorem \cites{marsden1974reduction} (a preliminary version can be found in Meyer \cites{meyer}) using the momentum mapping, a natural extension of the classical linear and angular momentum. The reduced manifold is obtained using a regular value of the momentum mapping and the corresponding isotropy group, and the dynamics is projected to this reduced manifold, gaining for integration a smaller number of degrees of freedom. This theorem has been extended to many other contexts: cosymplectic, contact, and more general settings (see \cites{acakpo2022stable, albert1989theoreme, de1993cosymplectic, de2019contact, garcia2022momentum, Lainztesis, marsden1990reduction,willett} and the references therein). For a recent review on reduction by symmetries in cosymplectic geometry we refer to \cites{delucas2023cosymplectic}.
When the reduced space is not a manifold, we can reduce the algebra of observables \cites{Sniatycki1983ReductionAQ} (see also \cites{beppe}). This reduction recovers the Poisson algebra of the reduced space in Marsden-Weinstein reduction. 

Related to the geometric reduction is Noether's theorem (in fact, this reduction is a generalisation of it), which states that a symmetry of the system produces a conserved quantity \cites{binz}. The introduction of geometric structures has revealed itself in a plethora of results relating symmetries and conserved quantities \cites{de2011methods}.


Furthermore, Lagrangian submanifolds play a crucial role, since it is easy to check that the image of a Hamiltonian vector field $X_H$ in a symplectic manifold $(M, \omega)$ can be interpreted as a Lagrangian submanifold of the symplectic manifold $(TM, \omega^c)$, where $\omega^c$ is the complete or tangent lift of $\omega$ to the tangent bundle $TM$. This result has its equivalent in Lagrangian mechanics, and has led to the so-called Tulczyjew triples, which elegantly relate the different Lagrangian submanifolds that appear in Lagrangian and Hamiltonian descriptions of mechanics via the Legendre transformation \cites{tulczyjew1976hamiltonienne,tulczyjew1976lagrangienne,de2011methods}. This interpretation of the dynamics as a Lagrangian submanifold has been extended to other scenarios, including the Tulczyjew triple \cites{deleon2003tulczyjews,jcelisa,ogul,ogul2,grabowska,ogul2021,silvia,willett}. Lagrangian submanifolds are also relevant to develop the so-called Hamilton--Jacobi theory since they provide the geometric setting for solutions of the Hamilton--Jacobi problem (see \cites{ogulin} for a recent topical review on the subject).
In this sense, we follow Weinstein's creed: "Everything is a Lagrangian submanifold" \cites{weinsteincreed}.

Moreover, coisotropic submanifolds play a relevant role both in the theory of constraints and in the theory of quantization. For instance, coisotropic submanifolds are precisely the first class constraints considered by Dirac in \cite{Dirac_1950} (see also \cites{BOJOWALD_2003}), where he developed the constraint algorithm for singular Lagrangians in the Hamiltonian setting. The constraint algorithm has been developed in geometrical terms in \cites{GotaySingularLagrangians1, GotaySinguarLagrangians2}. In this approach, the phase space of a singular system is a presymplectic manifold and, in \cite{got}, Gotay showed that every presymplectic manifold $(P, \Omega)$ may be imbedded in a symplectic manifold $(M, \omega)$ as a closed coisotropic submanifold. More precisely, there exists an imbedding $j: P \rightarrow M$ such that 
\begin{itemize}
\item $j(P)$ is closed in $M$;
\item $j^\ast \omega = \Omega$;
\item $TP^{\perp} \subseteq j_\ast (TP) $.
\end{itemize}
An alternative approach to the usual treatments of singular Lagrangians based on a Hamiltonian regularization scheme inspired on the coisotropic embedding of presymplectic systems was developed in \cites{regular}.

The ideas to develop the coisotropic reduction procedure came from Weinstein \cite{weinstein1977lectures} and were also partially inspired by Roels and Weisntein \cite{roels} and Marsden and Weinstein \cite{marsden1974reduction}.

Coisotropic reduction works when we give a coisotropic submanifold $N$ of a symplectic manifold $(M, \omega)$ and we consider (if it is well defined) the quotient manifold $N/(TN)^\perp $, where $(TN)^\perp$ is the symplectic complement of $TN$. Being involutive, this distribution along $N$ defines a foliation. The corresponding leaf space inherits a reduced symplectic form from the symplectic structure given on $M$. If in addition we have a Lagrangian submanifold $L$ with clean intersection with $N$, then $L \cap N$ projects into a Lagrangian submanifold of the quotient (see \cites{weinsteincreed, abraham2008foundations}). Coisotropic reduction can be also combined with symplectic reduction to develop a reduction procedure for the Hamilton-Jacobi equation in presence of symmetries (see \cites{vaquero}).

Coisotropic reduction has been extended to the field of contact manifolds (with the interest of being in a dissipative context) \cites{de2019contact, Tortorella2017}, but it has not been studied in sufficient detail in the case of cosymplectic manifolds nor in that of co-contact manifolds, the latter the natural settings to study time-dependent Hamiltonian contact systems \cites{de2022time,de2022time2}.

The objectives of this paper are twofold. On the one hand, to develop in detail the coisotropic reduction in the case of cosymplectic manifolds and those of co-contact, covering a gap in the literature. On the other hand, to present a survey that brings together in one place the different cases that appear in the study of Hamiltonian systems in classical mechanics.

The paper is structured as follows. Sections \ref{Symplecticvector} and \ref{Symplectic} are devoted to the main ingredients concerning symplectic Hamiltonian systems and the classical coisotropic reduction procedure. In order to go to the cosymplectic setting, we recall some general notions in Poisson structures (Section \ref{Poisson}) and then we consider the case of coisotropic reduction in the cosymplectic setting in Section \ref{Cosymplectic} (remember that this is the scenario to develop time dependent Hamiltonian systems). Contact manifolds require a more general notion than Poisson structures; indeed, they are examples of Jacobi structures, so that we give some fundamental notions in Section \ref{Jacobi}. The coisotropic reduction scheme developed in contact manifolds is the subject of Section \ref{Contact}, which is very different to the cosymplectic case since we are in presence of dissipative systems. We emphasize these differences in Section \ref{interlude}, where we study the corresponding Lagrangian settings. To combine dissipative systems with Hamiltonians depending also on time, we consider cocontact manifolds in Section \ref{Cocontact}, and develop there the corresponding coisotropic reduction procedure. Finally, we discuss a recent generalization of contact and cosymplectic systems called stable Hamiltonian systems in Section \ref{SHS}.

\section{Symplectic vector spaces} \label{Symplecticvector}

\setsectiontitle{SYMPLECTIC VECTOR SPACES}
We refer to \cites{abraham2008foundations, arnold, de2011methods, godbillon1969geometrie,libermann2012symplectic, vaisman2012lectures} for the main definitions and results.\\

\begin{Def}[Symplectic vector space] A \textbf{symplectic vector space} is a pair $(V, \omega)$ where $V$ is a finite dimensional vector space and $\omega$ is a non-degenerate $2$-form, called the \textbf{symplectic form}. Here, non-degeneracy means that the map $$ \flat_\omega: V \rightarrow V^*; \,\, v \mapsto i_v \omega$$ is an isomorphism.
\end{Def}

For every non-degenerate $2$-form
on $V$  there exist a basis $(x_i, y^i)$ whith $i$ taking values from $1$ to $n$ such that, making use of the summation convention, $\omega = x^i \wedge y_i$, where $(x^i, y_i)$ is the dual basis.
This implies that a symplectic vector space is necessarily of even dimension $2n$.\\

\begin{Def}[$\omega$-orthogonal] Let $W \subseteq V$ be a subspace of $V$. We define its \textbf{$\omega$-orthogonal} complement as $$W^{\perp_\omega} := \{v \in V \,\,| \,\, \omega(v,w) = 0, \,\, \forall w \in W\}.$$
\end{Def}

Note that $W^{\perp_\omega} = \ker (i^* \flat_\omega)$ where $i :W \hookrightarrow V$ is the natural inclusion. Using the non-degeneracy of $\omega$, this implies that $\dim W^{\perp_\omega} = \dim V - \dim W$, a result which will be useful throughout this paper.

The antisymmetry of $\omega$ gives rise to a wide variety of situations. In particular, we say that $W \subseteq V$ is:
\begin{itemize}
\item[$i)$] \textbf{Isotropic} if $W \subseteq W^{\perp_\omega}$ (if $W$ is isotropic, necessarily $\dim W \leq n$);
\item[$ii)$] \textbf{Coisotropic} if $W^{\perp_\omega} \subseteq W$ (if $W$ is coisotropic, necessarily $\dim W \geq n$);
\item[$iii)$] \textbf{Lagrangian} if $W$ is isotropic and has an isotropic complement (if $W$ is Lagrangian, necessarily $\dim W  =  n$);
\item[$iv)$] \textbf{Symplectic} if $V = W \oplus W^{\perp_\omega}$.
\end{itemize}

A subspace $W$ is Lagrangian if and only if $W = W^{\perp_\omega}$. This implies that Lagrangian subspaces are the isotropic subspaces of maximal dimension and the coisotropic subspaces of minimal dimension.

It can be easily checked that the symplectic complement has the following properties:
\begin{itemize}
\item[$i)$] $(W_1 \cap W_2)^{\perp_\omega} = W_1^{\perp_\omega} + W_2^{\perp_\omega};$
\item[$ii)$] $(W_1 + W_2)^{\perp_\omega} = W_1^{\perp_\omega} \cap W_2^{\perp_\omega};$
\item[$iii)$] $(W^{\perp_\omega})^{\perp_\omega} = W.$

\end{itemize}

\section{Coisotropic reduction in symplectic geometry}\label{Symplectic}
\setsectiontitle{COISOTROPIC REDUCTION IN SYMPLECTIC GEOMETRY}
\begin{Def}[Symplectic manifold]
 A \textbf{symplectic manifold} is pair $(M , \omega)$ where $M$ is a manifold and $\omega$ is
a closed $2$-form such that $(T_qM, \omega_q)$ is a symplectic vector space, for every $q \in M$. As in the linear case, for the existence of such a form, $M$ needs to have even dimension $2n$. 
\end{Def}

Every symplectic manifold is locally isomorphic, that is, there exists a set of canonical coordinates around each point:\\

\begin{theorem}[Darboux Theorem] \label{DarbouxTheorem}Let $(M, \omega)$ be a symplectic manifold and $q \in M$. There exist a coordinate system $(q^i,p_i)$ around $q$ such that $\omega = dq^i \wedge dp_ i$. These coordinates are called Darboux coordinates.
\end{theorem}

This non-degenerate form induces a bundle isomorphism between the tangent and cotangent bundles of $M$ point-wise, namely \begin{equation*} \flat_\omega: TM \rightarrow T^*M; \,\,\, v_q \mapsto \flat_\omega(v_q) = i_{v_q} \omega. \end{equation*} 
\begin{Def}[Hamiltonian vector field] Given $H \in \mathcal{C}^\infty(M)$, we define the \textbf{Hamiltonian vector field} of $H$ as $$X_H:= \sharp_\omega(dH),$$ where $\sharp_\omega = \flat_\omega^{-1}.$ We say that a vector field $X$ is \textbf{Hamiltonian} if $X = X_H$ for some function $H$ and say that $X$ is \textbf{locally Hamiltonian} if $X = X_H$ for some local function defined in a neighborhood of every point of the manifold.\\
\end{Def}

\begin{remark}{\rm
    Notice that a vector field is locally Hamiltonian if and only if $\flat_\omega(X)$ is closed, and Hamiltonian if and only if $\flat_\omega(X)$ is exact.}
\end{remark}

Locally, Hamiltonian vector fields have the expression $$X_H = \partder{H}{p_i} \partder{}{q^i} - \partder{H}{q^i} \partder{}{p_i}.$$
Then, the integral curves of the Hamiltonian vector field $X_H, (q^i(t), p_i(t))$, satisfy the local differential equations
\begin{align*}
    \totder{q^i}{t} &= \partder{H}{p_i},\\
    \totder{p_i}{t} &= - \partder{H}{q^i},
\end{align*}
which are the Hamilton's equations of motion.

The definitions of the different cases of subspaces given in the linear case can be extended point-wise to submanifolds $N \hookrightarrow M$. Consequently, we say that $N \hookrightarrow M$ is: 
\begin{itemize}
 \item[i)] \textbf{Isotropic} if $T_qN \subseteq T_qM$ is for every $q \in N$;
\item[ii)] \textbf{Coisotropic} if $T_qN \subseteq T_qM $ is for every $q \in N$;
\item[iii)] \textbf{Lagrangian} if $N$ is isotropic and there is a isotropic subbundle (where we understand isotropic point-wise) $E \subseteq TM | N$ such that $TM = TN \oplus E$ (here $\oplus$ denotes the Whitney sum). This is exactly the point-wise definition of a Lagrangian subspace asking for the coisotropic complement to vary smoothly;
\item[iv)] \textbf{Symplectic} if $T_qN \subseteq T_qM$ is for every $q \in N$.
\end{itemize}
These definitions extend naturally to distributions.

Just like in the linear case, a submanifold $N \hookrightarrow M$ is Lagrangian if and only if it is isotropic (or coisotropic) and has maximal (or minimal) dimension. This is a useful characterization that we will use several times in the rest of the paper.\\

\begin{lemma}\label{CharacterizationLagrangianSympelctic} Let $i: L \rightarrow M$ be a submanifold of dimension $n$. Then, $L$ is a Lagrangian submanifold of $(M, \omega)$ if and only if $i^*\omega = 0$.
\end{lemma}
\begin{proof} It is trivial, since Lagrangian submanifolds are the isotropic submanifolds of maximal dimension, say $n$.
\end{proof}

\subsection{Hamiltonian vector fields as Lagrangian submanifolds}
\begin{Def}Let $(M , \omega)$ be a symplectic manifold. Define the \textbf{tangent symplectic structure} on $TM$ as $\omega_0 = -d\lambda_0$ where $\lambda_0= \flat_\omega^*\lambda_M$, and $\lambda_M$ is the Liouville $1$-form on the cotangent bundle.
\end{Def}
Recall that $\lambda_M$ is defined as follows:
$$
\lambda_M(\alpha_x)(X_{\alpha_x}) = \alpha_x(d_{\alpha_x}\pi_M \cdot X_{\alpha_x})
$$
where $X_{\alpha_z} \in T_{\alpha_x}(T^*M)$, $\alpha_x \in T^*_xM$, and $\pi_M : T^*M \longrightarrow M$ is the canonical projection. The Liouville $1$-form can be also defined as the unique $1$-form $\lambda_M$ on $T^*M$ such that, for every $1$-form $\alpha:M \rightarrow T^*M$, $$\alpha^*(\lambda_M) = \alpha.$$

In coordinates $(q^i, p_i, \dot{q}^i, \dot{p}_i)$ the tangent symplectic structure is $$\omega_0 = -dq^i \wedge d\dot{p}_i - d\dot{q}^i \wedge dp_i.$$

\begin{prop}\label{Pullbackvectorfields} Let $X: M \rightarrow TM$ be a vector field. Then $$X^*(\lambda_0) = \flat_\omega(X).$$
\end{prop}
\begin{proof} This is a straight-forward verification. Let $v \in T_qM$, then we have
\begin{align*}
\langle X^*\lambda_0 ,v \rangle &= \langle \lambda_0,  d_qX \cdot v \rangle = \langle \lambda_M, d_{X(q)}\flat_\omega \cdot d_q X \cdot v \rangle \\
&= \langle \flat_\omega(X)^* (\lambda_M), v \rangle = \langle \flat_\omega(X), v \rangle.
\end{align*}
\end{proof}

\begin{prop} Let $X: M \rightarrow TM$ be a vector field. Then $X$ is locally Hamiltonian if and only if $X(M)$ is a Lagrangian submanifold of $(TM, \omega_0)$.
\end{prop}
\begin{proof} We only check that $X(M)$ is isotropic using \textcolor{red}{Lemma \ref{CharacterizationLagrangianSympelctic}}, since $\dim X(M) = \dim M  = \frac{1}{2} \dim TM.$ In fact:
\begin{align*}
X^*\omega_0 = -X^*(d\theta_0) = -d(X^*\theta_0) = -d(\flat_\omega(X)),
\end{align*}
which gives the characterization.
\end{proof} 

We can also check this last proposition easily in coordinates. Indeed, let $$X = X^i \partder{}{q^i} + Y_i \partder{}{p_i}.$$ We have $$-X^* \omega_0 =  \partder{Y_i}{q^j}  dq^i \wedge dq^j + \left( \partder{Y_i}{p_j} + \partder{X^j}{q_i} \right) dq^i \wedge dp_j+ \partder{X^i}{p_j} dp_j \wedge dp_i,$$
 and thus, $X$ defines a Lagrangian submanifold if and only if 
 \begin{align*}
  \partder{Y_i}{q^j} -  \partder{Y_j}{q^i}&= 0,\\
  \partder{Y_i}{p_j} + \partder{X^j}{q_i} &= 0,\\
  \partder{X^i}{p_j} - \partder{X^j}{p_i} &= 0.
 \end{align*}

 Taking $$(G^1, \dots, G^n, G^{n+1}, \dots, G^{2n}) := (X^i, -Y_i)$$ and $$(x^1, \dots, x^n, x^{n+1}, \dots, x^{2n}) := (q^i, p_i),$$ these conditions become $$\partder{G^i}{x^j} = \partder{G^j}{x^i}.$$ This implies $G^i = \displaystyle \partder{H}{x^i},$ for some local function $H$. It is clear that locally, we have $X = X_H$.
\subsection{Coisotropic reduction}

Now, given a coisotropic submanifold $N \hookrightarrow M$, we define the distribution $(TN)^{\perp_\omega}$ on $N$ as
the subbundle of $TM |_N$ consisting of all $\omega$-orthogonal spaces $(T_qN)^{\perp_\omega}$. Note that this distribution is regular and its rank is $\dim M  - \dim N$.\\

\begin{prop} Let $(M, \omega)$ be a symplectic manifold and $N \hookrightarrow M$ be a coisotropic submanifold.  The  distribution $q \mapsto (T_qN)^{\perp_\omega}$ is involutive.
\end{prop}
\begin{proof} Let $X,Y$ be vector fields along $N$ with values in
$TN^{\perp_\omega}$ and $Z$ be any other vector field tangent to $N$. Since $\omega$ is closed we have \begin{align*} 0 =&
( d\omega)(X,Y,Z) = X(\omega(Y,Z)) - Y(\omega(X,Z)) + Z\omega(X,Y)\\ & - \omega([X,Y], Z) + \omega([X,Z], Y) - \omega([Y,Z],X) = -
\omega([X,Y],Z), \end{align*} since $X,Y$ belong to the orthogonal complement of $TN$. We conclude that $\omega([X,Y],Z) = 0$ for every field $Z$ tangent to $N$, that is, $[X,Y] \in (TN)^{\perp_\omega}$. \end{proof}

Since the distribution is involutive and regular, the Frobenius Theorem guarantees the existence of a maximal regular foliation $\mathcal{F}$ of $N$, that is, a decomposition of $N$ into maximal submanifolds tangent to the distribution. In what follows, we suppose that $N / \mathcal{F}$ (the space of all leaves) admits a manifold structure so that the projection $$\pi : N \rightarrow N / \mathcal{F}$$ is a submersion. The main result is the Weinstein reduction theorem \cites{weinstein1977lectures}:\\ 

\begin{theorem}[Coisotropic reduction in the symplectic
setting] \label{SymplecticReduction} Let $(M , \omega)$ be a symplectic manifold and $N \hookrightarrow M$ be a coisotropic submanifold. If $N / \mathcal{F}$ (the spaces of all leaves under the distribution $q \mapsto (T_qN)^{\perp_\omega}$) admits a manifold structure such that $N \xrightarrow{\pi} N / \mathcal{F}$ is a submersion, there exist an unique $2$-form $\omega_N$ on $N / \mathcal{F}$ that defines a
symplectic manifold structure such that, if $N \xrightarrow{i}M$ is the natural inclusion, then $i^* \omega = \pi ^* \omega_N$. The following diagram summarizes the situation:
 \[
\begin{tikzcd} N \arrow[r, "i"] \arrow[d, "\pi"] & M \\ N / \mathcal{F} \end{tikzcd} \] \end{theorem}
 \begin{proof} Uniqueness is
guaranteed from the imposed relation since it forces us to define $$(\omega_N)_{[q]}([u],[v]):= \omega(u,v),$$ where $[u]:= T\pi(q) \cdot u$. We only need to check that this is a well-defined closed form and that it is non-degenerate.

We begin showing that our definition does not depend on the representative of the vector $[u]$. For this, it is sufficient to
observe that $(\omega_N)_{[q]}([u],[v]) = 0$ whenever $u$ is a vector in the distribution.

Furthermore, $$\mathcal{L}_X \omega = di_X \omega + i_X d\omega = 0$$ for every vector field $X$ in $N$ with values in
$(TN^{\perp_\omega})$, and this implies the independece of the point (for every two points in the same leaf of the foliation can be joined by a finite union of flows of such fields).

It is clearly non-degenerate and it is closed, since $d\pi^* \omega_N = i_*d\omega = 0$ and $\pi$ is a submersion. \end{proof}

\subsection{Projection of Lagrangian submanifolds}

\begin{Def}[Clean intersection]We say that two submanifolds $L, N \hookrightarrow M$ have \textbf{clean intersection} if $L
\cap N \hookrightarrow M$ is again a submanifold and $T_q(L \cap N) = T_q L \cap T_qN$, for every $q \in L \cap N$.\\ \end{Def}

 \begin{prop} \label{LagrangianProjectionSymplectic} Let $L
\hookrightarrow M$ be a Lagrangian submanifold and $N \hookrightarrow M$ a coisotropic submanifold. If they have clean
intersection and $L_N:= \pi(L \cap N)$ is a submanifold of $N/\mathcal{F}$, $L_N$ is Lagrangian. \end{prop} \begin{proof} It is
sufficient to see that is isotropic and that it has maximal dimension in $N / \mathcal{F}$. It is isotropic since $[u] \in
T_q(L _ N)$ implies $\omega_N([u],[v]) = \omega(u,v) = 0$, for every $[v] \in T_q(L _ N)$. Now, since $\ker d_q\pi =
(T_qN)^{\perp_\omega}$, the kernel-range formula yields \begin{equation}\label{eq1symplectic} \dim L_N = \dim(L \cap N) - \dim (T_q L \cap (T_qN)^{\perp_\omega}). \end{equation} 
Furthermore, \begin{equation}\label{eq2symplectic} \dim(L \cap N) + \dim (T_qL + (T_qN)^{\perp_\omega}) = \dim M,
\end{equation} beacause $L$ is Lagrangian and $N$ coisotropic. Substituting (\ref{eq2symplectic})in (\ref{eq1symplectic}) we obtain \begin{align} \nonumber
\dim L_N &= \dim M - \dim (T_qL + (T_qN)^{\perp_\omega})- \dim (T_q L \cap (T_qN )^{\perp_\omega}) \\ & \nonumber = \dim M - \dim L - \dim (T_qN)^{\perp_\omega} = \dim M - \dim L -(\dim M - \dim N)\\ &\nonumber = \dim N - \dim L = \dim N - \frac{1}{2} \dim
M, \end{align} which is exactly $\frac{1}{2} \dim N/ \mathcal{F}$, as a direct calculation shows. \end{proof}

\subsection{An example}
As an example of coisotropic reduction, let us take $$M := \mathbb{R}^{2(n +1)},$$ and $$N := \mathbb{S}^{2n +1}.$$ Since it has codimension $1$, it defines a coisotropic submanifold if we endow $M$ with the natural symplectic form $\omega :=  d q^i \wedge  d p_i.$  It is easy to check that $(TN)^{\perp_\omega}$ is generated by $$ \sum_i \left(p_i \partder{}{q^i} - q^i \partder{}{p_i}\right),$$
where $(q^i, p_i)$ are canonical coordinates in $\mathbb{R}^{2(n +1)}$.
Therefore, the leaves of the corresponding foliation are precisely the orbits of the previous vector field in $\mathbb{S}^{2n +1}$. Making the identification $$\mathbb{
R}^{2(n +1)} = \mathbb{C}^{n +1},$$ two points $x, y \in \mathbb{S}^{2n +1}$ are in the same orbit if and only if there exist some $\alpha \in \mathbb{C}$ ($|\alpha| = 1$) such that $$\alpha x = y.$$ 
This implies that $$\mathbb{S}^{2n +1}/\mathcal{F}$$ is the complex projective space of complex dimension $n $, $\mathbb{P}^n \mathbb{C}$. Therefore, we conclude that through coisotropic reduction we can define a natural symplectic structure on $\mathbb{P}^n \mathbb{C},$ for every $n$.

The above example is taken from Exercise 5.3B in \cites{abraham2008foundations}.

\section{Poisson structures}\label{Poisson}
\setsectiontitle{POISSON STRUCTURES}
A symplectic structure $(M,\omega)$ induces a Lie algebra structure on the ring of functions $\mathcal{C}^\infty(M)$.\\

\begin{Def}[Poisson bracket] Let $(M,\omega)$ be a symplectic manifold and $f, g \in \mathcal{C}^\infty(M)$. We define the Poisson bracket of $f,g$ as the function $$\{f,g\} := \omega(X_f,X_g).$$
\end{Def}

It is easily checked that in Darboux coordinates  the Poisson bracket is $$\{f,g\} = \frac{\partial f}{\partial q^i}\frac{\partial g}{\partial p_i}- \frac{\partial f}{\partial p_i} \frac{\partial g}{\partial q^i}.$$

\begin{prop} The Poisson bracket satifies the following properties:
\begin{itemize}
\item[i)] It is bilinear with respect to $\mathbb{R}$;
\item[ii)]  $\{f,g\cdot h\} = g \cdot \{f,h\} +\{f,g\} \cdot h $ (the Leibniz rule);
\item[iii)] $\{f,\{g,h\}\} + \{h,\{f,g\}\} + \{g,\{h,f\}\} = 0$ (the Jacobi identity).
\end{itemize}
\end{prop}

Taking into consideration the previous definition, we can generalize the notion of symplectic manifolds as follows:\\

\begin{Def}[Poisson manifold] A Poisson manifold is a pair $(P, \{ \cdot, \cdot \})$ where $P$ is a manifold and $\{ \cdot , \cdot \}$ is an antisymmetric bracket in the ring of functions $\mathcal{C}^\infty(P)$ satisfying the Leibniz rule and the Jacobi identity.\\
\end{Def}
\begin{Def}[Hamiltonian vector field, Characteristic distribution] Given $H \in \mathcal{C}^\infty(M),$ the Leibniz rule implies that $\{H,\cdot\}$ defines a derivation on $C^\infty(M)$ an thus is associated to a unique vector field $X_H$, which will be called the \textbf{Hamiltonian vector field} of $H$. The collection of all Hamiltonian vector fields generates the \textbf{characteristic distribution}, namely $$\mathcal{S}_q:= \langle v = X_H(q) \,\, \text{for all} \,\, H \in \mathcal{C}^\infty(M)\rangle.$$
\end{Def}

The definition of Hamiltonian vector field implies that $\{f,g\}$ only depends on the values of $df, dg$ and thus we can define a bivector field $$ \Lambda(\alpha, \beta):= \{f,g\}$$ where $df = \alpha, dg = \beta.$ We have $\{f,g\} = \Lambda(df,dg).$ $\Lambda$ also satisfies the partial differential equation $[\Lambda, \Lambda] = 0$, where $[\cdot, \cdot]$ is the Schouten-Nijenhuis bracket  \cites{vaisman2012lectures}. This last property is actually equivalent to the Jacobi identity, that is, given a bivector field $\Lambda$, $\{f,g\}:= \Lambda(df,dg)$ defines a Poisson structure if and only if $[\Lambda, \Lambda] = 0.$\\

\begin{Def} Let $(P, \{\cdot, \cdot \})$ be a Poisson manifold. Then we define $$\sharp_\Lambda: T^*P \rightarrow TP; \,\, \alpha_q \mapsto i_{\alpha_q}\Lambda.$$ Notice that $\im \sharp_\Lambda = \mathcal{S}$, the characteristic distribution.
\end{Def}

\medskip

\begin{remark}{\rm 
    In the case of symplectic manifolds $\sharp_ \omega = \sharp_\Lambda$, and the characteristic distribution is the whole tangent bundle; however, in the general setting $\sharp_\Lambda$ need not be a bundle isomorphism. Actually, if $\sharp_\Lambda$ is a bundle isomorphism, it arises form a symplectic structure defined as $\omega(v,w):= \Lambda(\sharp_\Lambda^{-1}(v), \sharp_\Lambda^{-1}(w))$ \cites{de2011methods}.}
\end{remark}

This type of distributions is in general not of constant rank, so we cannot directly apply the Frobenius theorem. But there is an extension of the result, due to Stefan \cites{stefan} and Sussmann \cite{sussmann1973orbits} (independently) that works for generalised distributions, locally generated by vector fields that leave the distribution invariant. This is the situation for characteristic distributions in the case of Poisson manifolds (see \cites{libermann2012symplectic}).

So, the characteristic distribution is involutive  and each leaf $S$ of the foliation admits a symplectic structure defining for $f,g \in \mathcal{C}^\infty(S)$ and $q \in S$, $$\{f,g\}(q) := \{\widetilde f, \widetilde g\}(q)$$ for arbitrary extensions $\widetilde f,\widetilde g \in\mathcal{C}^\infty(P)$ of $f,g$ respectively. It can be easily checked that this definition does not depend on the chosen functions and that it defines a non-degenerate Poisson structure and thus, $S$ is a symplectic manifold \cites{vaisman2012lectures}.
The symplectic form is given by $$\omega_S(X_f, X_g) = \{f,g\},$$ where $f, g \in C^\infty(S).$ \\

\begin{Def}[$\Lambda$-orthogonal] Let $\Delta_q \subseteq T_qP$ be a subspace on a Poisson manifold $(P, \{\cdot, \cdot \})$. We define the \textbf{$\Lambda$-orthogonal complement} $\Delta_q^{\perp_\Lambda}= \sharp_\Lambda(\Delta_q^0)$ where $\Delta_q^0$ is the annihilator of $\Delta_q$, that is, $\Delta_q^0:= \{\alpha \in T^*_qP \,\, | \, \, \alpha = 0 \,\, \text{in} \,\, \Delta_q\}.$
\end{Def}

Just as in the symplectic scenario, we say that a subspace $\Delta_q\subseteq T_qP$ is
\begin{itemize}
\item[$i)$] \textbf{Isotropic} if $\Delta_q \subseteq \Delta_q^{\perp_\Lambda}$ for every $q \in P$;
\item[$ii)$] \textbf{Coisotropic} if $\Delta_q^{\perp_\Lambda} \subseteq \Delta_q$ for every $q \in P$;
\item[$iii)$] \textbf{Lagrangian} if $\Delta_q =  \Delta_q^{\perp_\Lambda} \cap \mathcal{S}_q$ for every $q \in P$. Notice that this is equivalent to $\Delta_q \cap \mathcal{S}_q$ being Lagrangian in each symplectic vector space $\mathcal{S}_q.$
\end{itemize}

The $\Lambda$-orthogonal complement satisfies the following properties:

\begin{itemize}
\item[$i)$] $(W_1 \cap W_2)^{\perp_\Lambda} = W_1^{\perp_\Lambda} + W_2^{\perp_\Lambda};$
\item[$ii)$] $(W_1 + W_2)^{\perp_\Lambda} \subseteq  W_1^{\perp_\Lambda} \cap W_2^{\perp_\Lambda}.$\\
\end{itemize}

\begin{remark}{\rm
    For symplectic manifolds, the above definitions coincide with the ones previously given.}
\end{remark}


\section{Coisotropic reduction in cosymplectic geometry}\label{Cosymplectic}
\setsectiontitle{COISOTROPIC REDUCTION IN COSYMPLECTIC GEOMETRY}

Cosymplectic structures are relevant precisely because they are the natural arena to develop time-dependent Lagrangian and Hamiltonian mechanics \cites{de2011methods}.\\ 

\begin{Def}[Cosymplectic manifold] A \textbf{cosymplectic manifold} is a triple $(M,\Omega, \theta)$ where $M$ is a $(2n +
1)$-manifold, $\theta$ is a closed $1$-form and $\Omega$ is a closed $2$-form such that $\theta \wedge \Omega^n\neq 0$.
\end{Def} 
Similar to the symplectic setting, there exist canonical coordinates, which will be called Darboux coordinates $(q^i,p_i,t)$ such that $\Omega = dq^i \wedge dp_i$ and $\theta = dt$. The existence of such coordinate charts is proven in \cites{godbillon1969geometrie}.\\

There are two natural distributions defined on $M$: \begin{itemize} \item[i)] The \textbf{horizontal distribution} $\mathcal{H} :=
\ker \theta$; \item[ii)] The \textbf{vertical distribution} $\mathcal{V}:= \ker \Omega$. \end{itemize}
These distributions induce the following types of tangent vectors in each tangent space. A vector $v  \in T_qM$ will be called:

\begin{itemize}
\item[i)] \textbf{Horizontal} if $v \in \mathcal{H}_q$;
\item[ii)] \textbf{Vertical} if $v \in \mathcal{V}_q$.
\end{itemize}
In Darboux coordinates, these distributions are locally generated as follows: $$\mathcal{H} = \langle \frac{\partial}{\partial q^i}, \frac{\partial}{\partial p_i} \rangle; \,\, \mathcal{V} = \langle \frac{\partial}{\partial t} \rangle.$$

Just as before, we can define a bundle isomorphism between the tangent and cotangent bundles:
$$\flat_{\theta,\Omega}: TM \rightarrow T^*M;\,\,\, v_q \mapsto \flat_{\theta, \Omega}(v_q) = i_{v_q}\Omega +
\theta(v_q)\cdot \theta.$$
Its inverse is denoted by $\sharp_{\theta, \Omega}$.

The vector field defined as $\mathcal{R} := \sharp_{\theta, \Omega}(\theta)$ is called the \textbf{Reeb vector field}. The Reeb vector field is locally given by $$\mathcal{R}  = \frac{\partial}{\partial t}.$$

Let $H$ be a differentiable function on $M$. We define the following vector fields
\begin{itemize}
\item[i)] The \textbf{gradient vector field} $\grad H := \sharp_{\theta, \Omega} (dH)$;
\item[ii)] The \textbf{Hamiltonian vector field} $X_H := \grad H - \mathcal{R}(H) \mathcal{R}$;
\item[iii)] The \textbf{evolution vector field} $\mathcal{E}_H := X_H + \mathcal{R}.$
\end{itemize}

These vector fields have the local expressions:
\begin{align*}
\grad H &= \partder{H}{p_i} \partder{}{q^i} - \partder{H}{q^i} \partder{}{p_i} + \partder{H}{t} \partder{}{t},\\
X_H &=  \partder{H}{p_i} \partder{}{q^i} - \partder{H}{q^i} \partder{}{p_i}, \\
\mathcal{E}_H &=  \partder{H}{p_i} \partder{}{q^i} - \partder{H}{q^i} \partder{}{p_i} + \partder{}{t}.
\end{align*}
Notice that an integral curve $(q^i(\lambda), p_i(\lambda), t(\lambda)$ of the evolution vector field satisfies the following differential equations
\begin{align*}
    \totder{q^i}{ \lambda} &= \partder{H}{p_i},\\
    \totder{p_i}{\lambda} &= -\partder{H}{q^i},\\
    \totder{t}{\lambda} &= 1,
\end{align*}
which immediately give time-dependent Hamilton's equations
\begin{align*}
    \totder{q^i}{t} &= \partder{H}{p_i},\\
    \totder{p_i}{t} &= -\partder{H}{q^i}
\end{align*}
since we have $t=\lambda+const.$

Notice that the horizontal distribution $\mathcal{H}$ is the distribution generated by all Hamiltonian vector fields. Just as in the symplectic case, we can define a Poisson bracket:\\

\begin{Def}[Poisson bracket]Let $\{\cdot, \cdot \}$ be the bracket in the
ring $\mathcal{C}^{\infty}(M)$ given by $$\{f,g\} := \Omega(X_f,X_g).$$ \end{Def} 

We can easily check that this is indeed a Poisson structure by observing that in coordinates it is given by $$\{f,g\} = \frac{\partial f}{\partial q^i}\frac{\partial g}{\partial p_i}- \frac{\partial f}{\partial p_i} \frac{\partial g}{\partial q^i}.$$ Thus, the coordinate expression of the Poisson tensor is $$\Lambda = \frac{\partial}{\partial q^i} \wedge \frac{\partial}{\partial p_i}.$$

 So, the Hamiltonian vector fields coincide with the ones provided by this induced Poisson structure, following the notions and results given in \textcolor{red}{Section \ref{Poisson}}. In particular, given $\Delta_q \subseteq T_qM$,  we have $$\Delta_q^{\perp_\Lambda} = \sharp_{\Lambda}(\Delta_q^0).$$

Note that $\ker \sharp_\Lambda = \langle \theta \rangle $ and that $\im \sharp_\Lambda = \mathcal{H},$ that is, $\mathcal{H}$ is the characteristic 
distribution of the Poisson structure induced by $(\theta, \Omega)$. This implies the following result:\\
\begin{prop} $i: L \rightarrow M$ is a Lagrangian submanifold if and only if $$T_qL^{\perp_\Lambda} = T_qL  \cap \mathcal{H}_q$$ for every $q \in L.$
\end{prop}
\begin{proof} It follows from the definition of Lagrangian submanifold (\textcolor{red}{Section \ref{Poisson}}) and the fact that $\mathcal{H}$ is the characteristic distribution on $M$.
\end{proof}
It is also easy to see that $$\Lambda(\alpha, \beta) = \Omega(\sharp_{\theta, \Omega}(\alpha),\sharp_{\theta, \Omega}( \beta) )$$ observing that we have $\Omega(X_f,X_g) = \Omega(\grad f, \grad g)$. \\

\subsection[Hamiltonian vector fields as Lagrangian submanifolds]{Gradient, Hamiltonian and evolution vector fields as Lagrangian submanifolds}
\begin{Def} Given a cosymplectic manifold $(M, \Omega, \theta)$, we define the symplectic structure on $TM$ as $\Omega_0 := - d\lambda_0$, where $\lambda_0 = \flat_{\theta,\Omega}^*\lambda_M$, $\lambda_M$ being the Liouville $1$-form in the cotangent bundle $T^*M$.
\end{Def}

There is another expression of $\Omega_0$, namely $$\Omega_0 = -\Omega^c -\theta^c \wedge \theta^v$$ as one can verify \cites{cantrijn1992gradient}. Here, $\alpha^v, \alpha^c$ denote the complete and vertical lifts of a form $\alpha$ on $M$ to its tangent bundle $TM$ \cites{de2011methods}. This implies that in the induced coordinates in $TM$, $(q^i,p_i, t, \dot{q}^i, \dot{p}_i, \dot{t})$, $$\Omega_0 = -dq^i \wedge d\dot{p}_i - d\dot{q}^i \wedge dp_i - d \dot{t} \wedge dt.$$

\begin{prop} Let $(M, \Omega, \theta)$ be a cosymplectic manifold and $X: M \rightarrow TM$ a vector field. Then $X(M)$ is a Lagrangian submanifold of $(TM, \Omega_0)$ if and only if $X$ is locally a gradient vector field.
\end{prop}
\begin{proof}It is easily checked that $X^*\lambda_0 = \flat_{\theta,\Omega}(X)$ (just like in \textcolor{red}{Proposition \ref{Pullbackvectorfields}}) and then, $X(M)$ is Lagrangian if and only if $$0 = X^*\Omega_0 = -X^* d\lambda_0= -   d\flat_{\theta,\Omega}(X),$$ that is, $X$ is locally a gradient vector field.
\end{proof}

We can also check this in coordinates. Indeed, let $$X = X^i \partder{}{q^i} + Y_i \partder{}{p_i} + Z \partder{}{t}$$ be a vector field on $M$. $X: M \hookrightarrow TM$ defines a Lagrangian submanifold if and only if $X^*\Omega_0 = 0.$ An easy calculation gives 

\begin{align*}
    -X^*\Omega_0 =& \left ( \partder{X^i}{q^j} + \partder{Y_j}{p_i}\right ) dq^j \wedge dp_i +   \left ( \partder{X^i}{t} - \partder{Z}{p_i}\right ) dt \wedge dp_i +  \left ( \partder{Y_i}{t} + \partder{Z}{q^i}\right ) dq^i \wedge dt\\
   &  \partder{X^i}{p_j}  dp_j \wedge dp_i +  \partder{Y_i}{q^j} dq^j \wedge dq^i.
\end{align*}

Therefore, $X$ defines a Lagrangian submanifold of $(TM, \Omega_0)$ if and only if 

\begin{align}
    \label{1} \partder{X^i}{q^j} + \partder{Y_j}{p_i} = 0, \\ \label{2}\partder{X^i}{t} - \partder{Z}{p_i} = 0,\\ \label{3}\partder{Y_i}{t} + \partder{Z}{q^i} = 0, \\
    \label{4}\partder{X^i}{p_j} - \partder{X^j}{p_i} = 0, \\ \label{5}\partder{Y_i}{q^j} - \partder{Y_j}{q^i} = 0.
\end{align}

The equations above can be summarized by taking $$(G^1, \dots, G^n, G^{n+1} , \dots, G^{2n}, G^{2n +1}) := (X^i, - Y_i, Z),$$ $$(x^1, \dots, x^n, x^{n+1}, \dots, x^{2n}, x^{2n +1}) := (q^i, p_i, t),$$ since they translate to $$\partder{G^i}{x^j} = \partder{G^j}{x_i}.$$ We conclude that $G^i =  \displaystyle \partder{H}{x^i}$, for some local function $H$, that is, locally, $X = \grad H$.\\

In general, the Hamiltonian and evolution vector field do not define Lagrangian submanifolds in $(TM,\Omega_0)$. However, modifying the form we can achieve this. First, let us study the Hamiltonian vector field $X_H$. We have $$X_H^* \Omega_0 = (\grad H - \mathcal{R}(H) \mathcal{R}) ^* \Omega_0 = - d(\mathcal{R}(H) \theta) = - d(\mathcal{R}(H)) \wedge \theta.$$
The form defined as 
$$\Omega_H := \Omega_0 + \left( d\mathcal{R}(H) \wedge \theta \right)^v
$$ 
is a symplectic form and has the local expression  
$$
\Omega_H = - dq \wedge d\dot{p}_i - d\dot{q}^i \wedge dp_i - d\dot{t} \wedge dt + d\left( \partder{H}{t} \right) \wedge dt.
$$ 
Also,  $$X_H^* \Omega_H =  - d\mathcal{R}(H) \wedge \theta +  d(\mathcal{R}(H)) \wedge \theta = 0.$$ 
We have proved that $X_H$ defines a Lagrangian submanifold of $(TM, \Omega_H).$ Furthermore, since $$\mathcal{R} ^* \Omega_0 = 0,$$ it follows that the evolution vector field $\mathcal{E}_H$ also defines a Lagrangian submanifold of $(TM , \Omega_H)$. \\

This also gives a way of interpreting both vector fields as Lagrangian submanifolds of a the cosymplectic submanifold $(TM \times \mathbb{R}, \Omega_H, ds),$ taking the coordinate in $\mathbb{R}$ to be constant.

\subsection{Coisotropic reduction}

We can interpret the orthogonal complement defined by the Poisson structure using the cosymplectic structure. We note that $\Omega|_\mathcal{H}$ defined as $\Omega$ restricted to the distribution $\mathcal{H}$ induces a symplectic vector space in each $\mathcal{H}_q$ and thus we have a symplectic vector bunde $\mathcal{H} \rightarrow M$. If $\Delta_q \subseteq \mathcal{H}_q$, we have the $\Omega|_\mathcal{H}$-orthogonal complement $$(\Delta_q)^{\perp_{\Omega|_\mathcal{H}}} = \{ v \in \mathcal{H} \,\, | \,\, \Omega(v,w) = 0,\,\, \forall w \in \Delta_q\}.$$

\begin{prop}\label{CosymplecticOrthognal}Let $\Delta_q \subseteq T_qM$. Then $\Delta_q^{\perp_\Lambda} = (\Delta_q\cap \mathcal{H})^{\perp_{\Omega|_{\mathcal{H}}}}.$
\end{prop}
\begin{proof} Let $v \in \Delta_q^{\perp_\Lambda}$, that is, $v = \sharp_\Lambda(\alpha)$ with $\alpha \in \Delta_q^0$. This implies that $v$ is horizontal. We only need to check that $\Omega(v,w) = 0$ for every $w \in \Delta_q \cap \mathcal{H}_q$. Indeed, since $\theta(w) = 0$,
\begin{align*} 
\Omega(\sharp_\Lambda\alpha, w) &= - \Omega(w, \sharp_\Lambda\alpha) - \theta(w)\theta(\sharp_\Lambda \alpha) = -(\flat_{\theta, \Omega}w)(\sharp_\Lambda \alpha) = - \Lambda(\alpha, \flat_{\theta, \Omega} w)  \\
&= - \Omega(\sharp_{\theta, \Omega} \alpha, w)  =  - \Omega(\sharp_{\theta, \Omega} \alpha, w)  - \theta(\sharp_{\theta, \Omega}\alpha)\theta( w) = - \alpha(w) = 0.
\end{align*}
Now we compare dimensions. We distinguish two cases, if $\theta \in \Delta_q^0$, we have $$\dim \Delta_q^{\perp_\Lambda} = \dim \Delta_q^0 - 1 = 2n - \dim \Delta_q$$ which is exactly $\dim(\Delta_q \cap \mathcal{H}_q)^{\perp_{\Omega|_\mathcal{H}}}$, for $\Delta_q \subseteq \mathcal{H}_q$ and $(\mathcal{H}_q, \Omega |_\mathcal{H})$ is symplectic. Now, if $\theta \not \in \Delta_q^0$, then $$\dim \Delta_q^{\perp_\Lambda} = 2n +1 - \dim \Delta_q$$ and, since $\Delta_q \not \subseteq \mathcal{H}_q$, we have $\dim (\Delta_q \cap \mathcal{H}_q) = \dim \Delta_q - 1$ which implies that $\dim (\Delta_q \cap \mathcal{H}_q)^{\perp_{\Omega|_\mathcal{H}}} = 2n + 1 - \dim \Delta. $
\end{proof}

This last proposition clarifies the situation. The $\Lambda$-orthogonal of a subspace $\Delta$ is just the symplectic orthogonal of the intersection with the symplectic leaf. This means that coisotropic reduction in cosymplectic geometry will be performed in each leaf of the characteristic distribution $\mathcal{H}$. Also, because the $\Lambda$-orthogonal complement is just the symplectic complement of the intersection with $\mathcal{H}$, we have the following properties:
\begin{itemize}
\item[$i)$] $(\Delta_1 \cap \Delta_2)^{\perp_\Lambda} = \Delta_1^{\perp_\Lambda} + \Delta_2^{\perp_\Lambda}.$
\item[$ii)$] $(\Delta_1 + \Delta_2)^{\perp_\Lambda} = \Delta_1^{\perp_\Lambda} \cap \Delta_2^{\perp_\Lambda}.$
\item[$iii)$] $(\Delta^{\perp_\Lambda})^{\perp_\Lambda} = \Delta \cap \mathcal{H}.$
\end{itemize}

It will also be important to distinguish submanifolds $N \hookrightarrow M$ acording to the position relative to the distributions $\mathcal{H},\mathcal{V}$.\\
\begin{Def}[Horizontal, non-horizontal and vertical submanifolds] Let $i: N \hookrightarrow M$ be a submanifold. $N$ will be called a:
\begin{itemize}
\item[$i)$] \textbf{Horizontal submanifold} if $T_qN \subseteq \mathcal{H}_q$ for every $q \in N$;
\item[$ii)$] \textbf{Non-horizontal submanifold} if $T_qN \not \subseteq \mathcal{H}_q$ for every $q \in N$;
\item[$iii)$] \textbf{Vertical submanifold} if the Reeb vector field is tangent to $N$, that is, $\mathcal{R}(q) \in T_qN$ for every $q \in N$.\\
\end{itemize}

\begin{remark} {\rm Note that if $N\hookrightarrow M$ is a vertical submanifold, then $N$ is non-horizontal.}  
\end{remark}
\end{Def}

Lagrangian submanifolds are characterized as follows:\\
\begin{lemma}\label{PropertiesLagrangianCosympelctic}Let $L\hookrightarrow M$ be a Lagrangian submanifold and $q \in L$. Then
\begin{itemize}
\item[i)] If $T_qL \subseteq \mathcal{H}_q$, $\dim T_qL^{\perp_\Lambda} = \dim M - \dim L -1$.
\item[ii)] If $T_qL \not \subseteq \mathcal{H}_q$, $\dim T_qL^{\perp_\Lambda} = \dim M - \dim L$.
\end{itemize}
and, in either case, $\dim T_qL^{\perp_\Lambda} = n$, where $\dim M = 2n +1$.
\end{lemma}
\begin{proof}\item[$i)$] Since $\theta \in T_qL^0$ we have $$\dim T_qL^{\perp_ \Lambda} = \dim \sharp_\Lambda(T_qL^0) = \dim M - \dim L - \dim (\ker \sharp_\Lambda \cap T_qL) = \dim M - \dim L - 1.$$
\item[$ii)$] It follows from the previous calculation using that $\theta \not \in T_qL^0$ because $$\dim(\ker\sharp_\Lambda \cap T_qL) = 0.$$
The proof of the equality $\dim T_qL^{\perp_\Lambda} = n$ is straightforward using that $T_qL \cap \mathcal{H}_q$ is a Lagrangian subspace of $(\mathcal{H}_q, \Omega|_\mathcal{H})$.
\end{proof}

\textcolor{red}{Lemma \ref{PropertiesLagrangianCosympelctic}} guarantees that either $\dim L = n$, in which case $L$ is horizontal, or $\dim L = n+1$, in which case $L$ is non-horizontal. We have the following useful characterization of Larangian submanifolds:\\
\begin{lemma} \label{CharacterizationLagrangianCosymplectic}Let $L \hookrightarrow M$ be a submanifold. We have
\begin{itemize}
\item[i)] If $\dim L = n$, then $L$ is Lagrangian if and only if $i^* \theta = 0, i^* \Omega = 0$.
\item[ii)] If $L$ is non-horizontal and $\dim L = n+1$, then $L$ is Lagrangian if and only if $i^*\Omega = 0$.
\end{itemize}
\end{lemma}
\begin{proof} Both assertions are proved by a comparison of dimensions.
\end{proof}

\begin{prop} Let $ i:N \hookrightarrow M $ be a coisotropic submanifold. Then the
distribution $(TN)^{\perp_\Lambda}$ is involutive. \end{prop} 

\begin{proof} We start proving that $\mathcal{H}$ is an involutive distribution. Let $X,Y$ be vector fields tangent to $\mathcal{H}$. Since $\theta$ is closed we have 
\begin{align*}
0 &= (d\theta)(X,Y) = X(\theta(Y)) - Y(\theta(X)) - \theta([X,Y]) = - \theta([X,Y]),
\end{align*}
that is, $[X,Y]$ is tangent to $\mathcal{H}.$\\

 Denote $\Omega_0:= i^* \Omega$.
 Let $X,Y$ be vector fields in $N$ tangent to $(TN)^{\perp_\Lambda}$. Using \textcolor{red}{Proposition \ref{CosymplecticOrthognal}}, $[X,Y] \in (TN)^{\perp_\Lambda}$ if and only if $[X,Y] \in (TN \cap \mathcal{H})^{\perp_{\Omega |_\mathcal{H}}}$. In order to see this,  we take an arbitrary vector field $Z$ on $N$ tangent to $\mathcal{H}$ and check that $\Omega_0([X,Y],Z) = 0$. Because $\Omega$ is closed, we have
\begin{align*} 0 =& i^*( d\Omega)=  ( d\Omega_0)(X,Y,Z) = X(\Omega_0(Y,Z)) - Y(\Omega_0(X,Z)) + Z( \Omega_0(X,Y)) \\
&- \Omega_0([X,Y],Z) + \Omega_0([X,Z],Y) - \Omega_0([Y,Z],X) = - \Omega_0([X,Y],Z)\end{align*}
 where we have used that $X,Y,Z,[Y,Z], [X,Z]$ are horizontal (since $\mathcal{H}$ is involutive) and that $X,Y \in  (TN \cap \mathcal{H})^{\perp_{\Omega |_\mathcal{H}}}$. \end{proof}

\subsection{Vertical coisotropic reduction}
 We shall now study coisotropic reduction of a \textbf{vertical submanifold} $N \hookrightarrow M$. Let $q \in N$. We have $\dim (T_qN)^0 = \dim M - \dim N$. Since $N$ is vertical, $\theta \not \in (T_qN)^0$ and we have $$\dim (T_qN)^{\perp_\Lambda} = \dim \sharp_\Lambda(T_qN)^0 = \dim M - \dim N - \dim (\ker \sharp_\Lambda \cap (T_qN)^0) = \dim M - \dim N.$$ In particular, $(TN)^{\perp_\Lambda}$ is a regular distribution.\\

\begin{theorem}[Vertical coisotropic reduction  in the cosymplectic setting] \label{VerticalCosymplecticReduction} Let $(M, \Omega, \theta)$ be a cosymplectic manifold and $i:N \hookrightarrow M$ be an coisotropic vertical submanifold. Denote by $\mathcal{F}$ the maximal foliation of the involutive regular distribution $(TN)^{\perp_\Lambda}$. If the space of all leaves $N/ \mathcal{F}$ admits a manifold structure such that the projection $\pi : N \rightarrow N/ \mathcal{F}$ is a submersion, then there exist unique $\theta_N$, $\omega_N$ such that
\begin{align*} i^* \omega &= \pi^* \omega_N,\\ i^* \theta &= \pi^* \theta_N, \end{align*} and they define a cosymplectic structre on
$N / \mathcal{F}.$ The following diagram summarizes the situation:
\[ \begin{tikzcd} N \arrow[r, "i"] \arrow[d, "\pi"] & M \\ N / \mathcal{F} \end{tikzcd} \] \end{theorem}

 \begin{proof} Uniqueness is clear from the imposed relation. Denote $\Omega_0:= i^* \Omega$, $\theta_0 := i^*\theta$. We only need to verify that the following forms are closed, well defined and define a cosymplectic structure: \begin{align*} \Omega_N([u],[v]) := \Omega_0(u,v),\\
\theta_N([u]):= \theta_0(u), \end{align*} where $[u]:= T\pi(q) \cdot u \in T_{[q]}N/ \mathcal{F}$. If they were well defined, it is clear that 
they are smooth and closed since $\pi^* d \theta_N = d \theta_0 = 0$, $\pi^*\Omega_N =  d \Omega_0 = 0$ and $\pi$ is a submersion.\\ 

Let us first check that these definitions do not depend on the chosen representatives of the vectors. It suffices to observe that
for vectors in the distribution, say $v \in (T_qN)^{\perp_\Lambda}$, we have $i_v \Omega_0 = 0$ and $i_v \theta_0 = 0$. This easily follows from \textcolor{red}{Proposition \ref{CosymplecticOrthognal}} using that the horizontal proyection of every vector $u \in T_qN$ is tangent to $N$ (here we use the condition $\mathcal{R}(q) \in T_qN$). To see the independence of the point in the leaf chosen, it is enough to observe that $$\mathcal{L}_X \Omega_0 = 0; \,\,\, \mathcal{L}_X \theta_0 = 0 $$ for every vector field on $N$ tangent to the distribution $(T_qN)^{\perp_\Lambda}$ (since every two points in the same leaf of the foliation
can be joined by a finite union of flows of such fields).  Indeed, we have
\begin{align*}
\mathcal{L}_X\Omega_0 &= i_X d\Omega_0 + di_X \Omega_0 = 0,\\
\mathcal{L}_X \theta_0  &= i_X d\theta_0 + di_X\theta_0 = 0.
\end{align*}

Now we check that they define a cosymplectic structure. Assuming $k = \dim N$ and $2n +1 = \dim M$, from the remark above we have $$\dim N/\mathcal{F} = \dim N - \operatorname{rank}\, (TN)^{\perp_\Lambda} = 2k - 2n -1 = 2(k - n -1 ) + 1$$ and hence, $(N/\mathcal{F},\Omega_N,\theta_N)$ is a cosymplectic manifold if and only if $$\theta_N \wedge \Omega_N^{k-n-1}\neq 0,$$ which is equivalent to $\theta_0 \wedge \Omega_0^{k-n-1} \neq 0$, because $\pi$ is a submersion. For every point $q \in N$, $TqN$ can be decomposed in $$TqN =  TqN_\mathcal{H} \oplus \mathcal{V}$$ where $T_qN_\mathcal{H} = (T_qN) \cap \mathcal{H}_q$. It is easy to see that $(T_qN)_\mathcal{H}$ is a coisotropic subspace of $(\mathcal{H}_q, \Omega|_\mathcal{H})$. This implies (using symplectic reduction) that there are $\dim(T_qN)_\mathcal{H} - \dim (T_qN)_\mathcal{H}^{\perp_{\widetilde \Omega}} = k - 1 - (2n +1 -k) = 2k - 2n - 2$ horizontal vectors, say $u_1, \dots, u_{2k-2n-2}$ such that $\Omega_0^{k-n-1}(u_1, \dots, u_{2k-2n-2})\neq 0$. Taking the last vector to be $\mathcal{R}(q)$, it is clear that $$(\theta_0 \wedge \Omega_0^{k-n-1})(\mathcal{R}(q), u_1, \dots, u_{2k-2n-2}) \neq 0.$$ \end{proof}

\subsubsection{Projection of Lagrangian submanifolds}
Now we will prove that Lagrangian submanifolds $L \hookrightarrow M$ project to Lagrangian submanifolds in $N/\mathcal{F}.$\\

\begin{prop}[Projection of horizontal Lagrangian submanifolds is Lagrangian]\label{ProjectionHorizontalLagrangianCosymplectic} Under the hypotheses of \textcolor{red}{Theorem \ref{VerticalCosymplecticReduction}}, let $L \hookrightarrow M$ be an horizontal Lagrangian submanifold such that $L$ and $N$ have clean intersection. If $ L_N:=  \pi(L \cap N)$ is a submanifold of $N/ \mathcal{F}$, then $L_N$ is Lagrangian.
\end{prop}
\begin{proof} Let $\mathcal{H}_N$ be the horizontal distribution in $N/\mathcal{F}$. It is clear that $L_N$ is horizontal, because $T_{[q]}L_N = T\pi(q)(T_qL \cap T_qN)$. Using \textcolor{red}{Proposition \ref{CosymplecticOrthognal}} we have $$T_{[q]}L_N^{\perp_{\Lambda_N}} = (T_{[q]}L_N \cap \mathcal{H}_ N)^{\perp_{\Omega_N|_{\mathcal{H}_N}}}= T_{[q]}L_N^{\perp_{\Omega_N|_{\mathcal{H}_N}}}.$$ We will check that $$T_{[q]}L_N \subseteq T_{[q]}L_N^{\perp_{\Lambda_N}}$$ and prove that $\dim L_N = \dim N- n - 1$, which with \textcolor{red}{Lemma \ref{CharacterizationLagrangianCosymplectic}} together with the calculation of the dimension of $N / \mathcal{F}$ done in \textcolor{red}{Theorem \ref{VerticalCosymplecticReduction}}, yields the result. \\

Let $[v], [w] \in T_{[q]}L_N$. Then $$\Omega_N([v],[w]) = \Omega(v,w) = 0,$$ since $L$ is Lagrangian. Because $[w]$ is arbitrary, this last calculation implies that $[v] \in T_{[q]}L_N^{\perp_{\Omega_N|_{\mathcal{H}_N}}} \subseteq T_{[q]}L_N^{\perp_{\Lambda_N}}$. Now, 
\begin{equation}\label{eq1cosymplectic}
\dim L_N = \dim (L \cap N) - \dim (T_qL \cap (T_qN)^{\perp_\Lambda}).
\end{equation}
Furthermore, since $L$ is Lagrangian and horizontal, $(T_qL \cap (T_qN)^{\perp_\Lambda})^{\perp_\Lambda} = T_qL \cap \mathcal{H}_q + (T_qN^{\perp_\Lambda})^{\perp_\Lambda} = T_qL + (T_qN^{\perp_\Lambda})^{\perp_\Lambda}$ and thus (using that $T_qL \cap (T_qN^{\perp_\Lambda})^{\perp_\Lambda}$ is necessarily horizontal),
\begin{equation*}
\dim (T_qL \cap (T_qN)^{\perp_\Lambda}) = \dim M - \dim(T_qL + (T_qN^{\perp_\Lambda})^{\perp_\Lambda}) - 1.
\end{equation*}
Since $\dim(T_qL\cap \mathcal{H}_q + (T_qN^{\perp_\Lambda})^{\perp_\Lambda}) = \dim(T_qL\cap \mathcal{H}_q + T_qN) - 1$ (which comes from the fact that $\dim (T_qN^{\perp_\Lambda})^{\perp_\Lambda} = \dim T_qN - 1$) we have
\begin{equation}\label{eq2cosymplectic}
\dim(T_qL \cap \mathcal{H}_q + T_qN^{\perp_\Lambda}) = \dim M - \dim(T_qL \cap T_qN^{\perp_\Lambda}).
\end{equation}
Substituting (\ref{eq2cosymplectic}) in (\ref{eq1cosymplectic}) and using $\dim L = n$, we conclude 
\begin{align*}
\dim L_N &= \dim (L \cap N) - \left ( \dim M - \dim(T_qL + T_qN)\right) \\
	&=  \dim (L \cap N) - \dim M   + \dim L + \dim N - \dim(L \cap N)\\
	&= -2n -1 + n + \dim N = \dim N - n - 1.
\end{align*}
\end{proof}

\begin{prop}[Projection of non-horizontal Lagrangian submanifold is Lagrangian] \label{ProjectionNonHorizontalLagrangian}  Under the hypotheses of \textcolor{red}{Theorem \ref{VerticalCosymplecticReduction}}, let $L \hookrightarrow M$ be a non-horizontal Lagragian submanifold. If $L$ and $N$ have clean intersection and $L_N:= \pi(L \cap N) \hookrightarrow N/\mathcal{F}$ is a submanifold, then $L_N$ is Lagrangian.
\end{prop}
\begin{proof} The proof follows the same lines as that of \textcolor{red}{Proposition \ref{ProjectionHorizontalLagrangianCosymplectic}}. That $L_N$ is isotropic follows easily from \textcolor{red}{Proposition \ref{CosymplecticOrthognal}}. However, in order to calculate $\dim L_N$, we need to distinguish whether $L \cap N$ is horizontal or not. 
\item[$i)$] If $L\cap N$ is horizontal, we need to check that $\dim L_N = \dim N - n - 1$, since $L_N$ is horizontal. Because $(T_qL \cap T_qN^{\perp_\Lambda})^{\perp_\Lambda} = T_qL \cap \mathcal{H}_q + (T_qN^{\perp_\Lambda})^{\perp_\Lambda}$, we have
\begin{align*}
 \dim(T_qL\cap \mathcal{H}_q + (T_qN^{\perp_\Lambda})^{\perp_\Lambda}) &= \dim M - \dim (T_qL \cap T_qN^{\perp_\Lambda}) - 1.
\end{align*}
It is easy to check that $\dim(T_qL\cap \mathcal{H}_q + (T_qN^{\perp_\Lambda})^{\perp_\Lambda}) = \dim(T_qL\cap \mathcal{H}_q + T_qN) - 1$ and thus, $$\dim(T_qL \cap T_qN^{\perp_\Lambda}) = \dim M - \dim (T_qL\cap \mathcal{H}_q + T_qN).$$ We conclude that 
\begin{align*}
\dim L_N &= \dim(L \cap N) - \dim (T_qL \cap T_qN^{\perp_\Lambda})\\
& = \dim (L \cap N) - \left(\dim M - \dim(T_qL \cap \mathcal{H}_q  +T_qN) \right) \\
&=  \dim(L \cap N) - \left(\dim M - \dim(T_qL \cap \mathcal{H}_q) - \dim N+ \dim(T_qL \cap \mathcal{H}_q \cap T_qN)\right)\\
&= \dim(L \cap N) - \dim M + \dim L - 1 + \dim N -  \dim(T_qL \cap T_qN)\\
&= \dim N - \dim M + \dim L - 1 = \dim N - 2n - 1 + n + 1 - 1 = \dim N - n -1,
\end{align*}
where we have used that $T_qL \cap \mathcal{H}_q \cap T_qN = T_qL \cap T_qN$, since $N \cap L$ is horizontal, and $\dim T_qL \cap \mathcal{H} _q= \dim L  - 1$, because $T_qL \not \subseteq \mathcal{H}_q$.

\item[$ii)$] If $L\cap N$ is not horizontal, we need to check that $\dim L_N = \dim N - n$. This follows from the same calculation done in i), using that $$\dim(T_qL \cap \mathcal{H}_q \cap T_qN) = \dim(T_qL \cap T_qN) -1.$$
\end{proof}

\subsection{Horizontal coisotropic reduction}
We will restrict the study to \textbf{horizontal} coisotropic submanifolds $N\hookrightarrow M$, that is, manifolds satisfying $T_qN \subseteq \mathcal{H}_q$ for every $q\in N$. Note that in this case the distribution $(TN)^{\perp_\Lambda}$ is also regular, since $$\dim (T_qN)^{\perp_\Lambda} = \dim M - \dim N - \dim (\ker \sharp_\Lambda\cap (T_qN)^0)= \dim M - \dim N - 1.$$

\begin{theorem}[Horizontal coisotropic reduction in the cosymplectic setting] Let $(M, \Omega, \theta)$ be a cosymplectic manifold and $i: N \hookrightarrow M$ be an horizontal coisotropic submanifold. Denote by $\mathcal{F}$ the space of leaves determined by the regular and involutive distribution $(TN)^{\perp_\Lambda}.$ If $N/ \mathcal{F}$ admits a manifold structure such that $\pi : N \rightarrow N/ \mathcal{F}$ is a submersion, then there exists a unique $2$- form $\Omega_N$ in $N/\mathcal{F}$ such that $$\pi^* \Omega_N = i^* \Omega$$and  $(N/\mathcal{F}, \Omega_N)$ is a symplectic manifold. 
\end{theorem}
\begin{proof} Since $N$ is horizontal and the horizontal distribution is integrable, $N$ will be contained in an unique symplectic leaf and thus, we are performing symplectic reduction. The proof is just repeating what has been done in \textcolor{red}{Theorem \ref{SymplecticReduction}}.
\end{proof}

We can generalize this process to arbitrary submanifolds. Let $N \hookrightarrow M$ be a coisotropic submanifold. Since in general we cannot guarantee the well-definedness of the $2$-form in the quotient, we will reduce the intersection of $N$ with each one of the symplectic leaves. It is clear that $TN\cap \mathcal{H}$ is an involutive distribution, since $TN$ and $\mathcal{H}$ are. If this distribution was regular,  for every $q \in N$ there would exist an unique maximal leaf of the distribution, say $S_q$. Notice that $S_q \hookrightarrow M$ is an horizontal submanifold. We can perform coisotropic reduction on each of these submanifolds.

\subsubsection{Projection of Lagrangian submanifolds}

\begin{prop} Let $L \hookrightarrow M$ be a Lagrangian submanifold. If $L$ and $N$ have clean intersection and $L_N:= \pi(L\cap N)$ is a submanifold, then $L_N$ is Lagrangian.
\end{prop}
\begin{proof} Let $q \in N \cap L$, we have to prove that $$T_{[q]} L_N = d_q\pi  \cdot \left( T_q L \cap T_q N\right)$$ is a Lagrangian subspace of $T_{[q]} N/\mathcal{F} = T_{[q]}N / \left( T_{q} N\right)^{\perp_{\Lambda}}.$ Since $N$ is horizontal, $T_q N \subseteq \mathcal{H}_q$. Now, from \textcolor{red}{Proposition \ref{CosymplecticOrthognal}}, we know that $$(T_qN) ^{\perp_{\Omega|_{\mathcal{H}}}} = (T_q N)^{\perp_\Lambda} \subseteq T_qN,$$ that is, $T_qN$ is a coisotropic subspace of $\mathcal{H}_q$, with its natural symplectic structure. A similar argument shows that $T_qL$ is a Lagrangian subspace of $\mathcal{H}_q.$ Now, from symplectic reduction (linear symplectic reduction) we conclude that $$ T_{[q]} L_N = d_q\pi  \cdot \left( T_q L \cap T_q N\right)$$ is a Lagrangian submanifold of $T_qN/\left( T_{q} N\right)^{\perp_{\Lambda}} = T_{[q]} N/ \mathcal{F}$, with the symplectic structure induced by $\Omega_q|_{\mathcal{H}}$, which coincides sith the symplectic structure induced by $\Omega_q$, as a quick check shows.
\end{proof}

\section{Jacobi structures} \label{Jacobi}
\setsectiontitle{JACOBI STRUCTURES}
Contact and cocontact manifolds are not Poisson manifolds. However, there is still a Lie bracket defined in the algebra of functions, as we will see. This bracket induces what is called a Jacobi manifold. In this section we define and study such structures (see \cites{libermann2012symplectic, raul} for more details).\\

\begin{Def}[Jacobi Manifold] A \textbf{Jacobi structure} on a manifold $M$ is a Lie bracket defined in the algebra of functions $(\mathcal{C}^\infty(M), \{ \cdot, \cdot \})$ that satisfies the weak Leibniz rule, that is, $$\supp\{f,g\} \subseteq \supp f \cap \supp g.$$
\end{Def}

Every Jacobi bracket can be uniquely expressed as $$\{f,g\} = \Lambda(df, dg) + f E(g) - gE(f),$$ where $\Lambda$ is a bivector field (called the \textbf{Jacobi tensor}) and $E$ is a vector field. $\Lambda$ and $E$ satisfy the equalities $$[E, \Lambda] = 0, \,\, [\Lambda, \Lambda] = 2E \wedge \Lambda;$$ where $[\cdot, \cdot]$ is the Schouten-Nijenhuis bracket. Conversely, given a bivector field $\Lambda$ and a vector field $E$, $$\{f,g\} := \Lambda(df, dg) + f  E(g) - g E(f)$$ defines a Jacobi bracket if and only if both equalities above hold.\\

\begin{remark} {\rm
It is clear that Poisson manifolds are Jacobi manifolds, taking $E = 0$.}
\end{remark}

The Jacobi tensor allows us to define the morphism $$\sharp_\Lambda: T^*M \rightarrow TM;\,\, \alpha \mapsto i_{\alpha} \Lambda.$$ Define the $\Lambda$-orthogonal of distributions $\Delta$ as $$\Delta^{\perp_\Lambda} := \sharp_\Lambda(\Delta^0).$$
We can define the Hamiltonian vector field defined by a function $H$ as
$$
X_H = \sharp_{\Lambda}(dH) + H E
$$
Just like in the Poisson case, we say that a distribution $\Delta$ is:
\begin{itemize}
    \item[$i)$]\textbf{Isotropic} if $\Delta \subseteq \Delta^{\perp_\Lambda}$;
    \item[$ii)$]\textbf{Coisotropic} if $\Delta^{\perp_\Lambda} \subseteq \Delta$;
    \item[$iii)$]\textbf{Legendrian} if $\Delta^{\perp_\Lambda} = \Delta$.
\end{itemize}

These definitions extend naturally to submanifolds.\\

\begin{remark}{\rm 
    As in the case of Poisson manifolds, a Jacobi structure on a manifold $M$ defines a characteristic distribution ${\cal S}$ as follows:
    ${\cal S}_x$ is the vector subspace of $T_xM$ generated by the values of all Hamiltonian vector fields at $x$ and the vector field $E$ evaluated at $x$. This is again an involutive distribution in the sense of Stefan and Sussmann \cites{stefan,sussmann1973orbits}, and the leaves of the corresponding foliation are contact manifolds if the leaf has odd dimension, and locally conformal symplectic manifolds, if the leaf has even dimension. The definition of the Jacobi bracket on the leaves follows the same path that in the case of Poisson manifolds.}
\end{remark}
\section{Coisotropic reduction in contact geometry} \label{Contact}
\setsectiontitle{COISOTROPIC REDUCTION IN CONTACT GEOMETRY}

Contact manifolds are the natural setting for Hamiltonian systems with dissipation, instead of symplectic Hamiltonian systems where the antisymmetry of the symplectic form provides conservative properties \cites{bravetti0,de2019contact}. In the Lagrangian picture, contact Lagrangian systems correspond to the so-called Lagrangians depending on the action, and instead of Hamilton's principle, one has to use the so-called Herglotz principle to obtain the dynamics \cites{miguel}.\\

\begin{Def}[Contact manifold] A
contact manifold is a couple $(M, \eta)$ where $M$ is a $(2n+1)$-dimensional manifold, $\eta$ is a $1$-form and $\eta \wedge ( d\eta)^n \neq 0$. \end{Def} 
 In this case we also have Darboux coordinates $(q^i,p_i,z)$ in $M$  \cites{godbillon1969geometrie} such that $$\eta = dz - p_idq_i; $$
We have also have a bundle isomorphism defined as in the cosymplectic case $$\flat_{\eta}: TM \rightarrow T^* M ; \,\, v_q \mapsto
i_{v_q}d\eta + \eta(v_q)\cdot \eta,$$ its inverse $\sharp_\eta = \flat_ \eta^{-1}$, and a couple of natural distributions:
\begin{itemize}
\item[$i)$] The \textbf{horizontal} distribution $\mathcal{H}:= \ker \eta$;
\item[$ii)$] The \textbf{vertical} distribution $\mathcal{V}:= \ker d\eta$.
\end{itemize}
We can find different types of tangent vectors at a point $q \in M$. Indeed, a tangent vector $v \in T_qM$ will be called
\begin{itemize}
\item[$i)$] \textbf{Horizontal} if $v \in \mathcal{H}_q$;
\item[$ii)$] \textbf{Vertical} if $v \in \mathcal{V}_q$.
\end{itemize}

This time, however, we cannot define a canonical Poisson structure since the bivector field $$\Lambda(\alpha,\beta):= -d\eta(\sharp_{\eta}(\alpha), \sharp_{\eta}(\beta))$$ is not a Poisson tensor. In fact, $$[\Lambda,\Lambda] = - 2 \mathcal{R} \wedge \Lambda; \,\, [E, \Lambda] = 0$$ where $\mathcal{R}$ is the \textbf{Reeb vector field} defined as $\mathcal{R} := \sharp_{\eta}(\eta)$ (locally $\mathcal{R} = \frac{\partial}{\partial z}$). This is easily seen performing a direct calculation in Darboux coordinates using the local expresion $$\Lambda = \frac{\partial}{\partial p_i}\wedge \frac{\partial}{\partial q^i} + p_i \frac{\partial}{\partial p_i} \wedge \frac{\partial}{\partial z}.$$

 This defines a Jacobi structure in $M$ taking $\Lambda$ as above and $E = - \mathcal{R}$ (see \textcolor{red}{Section \ref{Jacobi}}). The Jacobi bracket is locally expressed as $$\{f,g\} =  \frac{\partial f}{\partial p_i}\frac{\partial g}{\partial q^i}- \frac{\partial f}{\partial q^i} \frac{\partial g}{\partial p_i} + p_i\left (  \frac{\partial f}{\partial p_i}\frac{\partial g}{\partial z}- \frac{\partial g}{\partial p_i} \frac{\partial f}{\partial z} \right ) + g \frac{\partial f}{\partial z} - f\frac{\partial g}{\partial z}.$$

The morphism induced by the Jacobi tensor $\Lambda$ satisfies $$\ker \sharp_\Lambda = \langle  \eta\rangle, \,\,\,\, \im \sharp_\Lambda = \mathcal{H}.$$

\begin{remark}{\rm Although notation between cosymplectic and contact manifolds is similar, they are different in nature. In cosymplectic geometry, we had a closed $1$-form $\theta$, and a closed $2$-form $\Omega$ satisfying the non-degeneracy condition. In contact geometry we have a $1$-form $\eta$ and a (closed) $2$-form $d \eta$ which also satisfies the non-degeneracy condition. In each tangent space, these structures will be isomorphic. Indeed, we can always assume that $p_i = 0$ at certain $q \in M$, which would give $\eta = dz$ in said point. However, they are far from locally isomorphic. In cosymplectic manifolds, the horizontal distribution is involutive, but in contact manifolds is not (this is the key to obtain dissipative dynamics). In \textcolor{red}{Section \ref{interlude}} we will enphasize on the differences between these geoemtries by studying the variational principle in contact mechanics.}
\end{remark}

\subsection[Hamiltonian vector fields as Lagrangian submanifolds]{Hamiltonian and evolution vector fields as Lagrangian and Legendrian submanifolds}
\begin{Def}[Hamiltonian vector field] Let $H \in \mathcal{C}^\infty(M)$. Define the \textbf{Hamiltonian vector field} of $H$ as $$X_H := \sharp_ \Lambda(dH) - H \mathcal{R} = \sharp_\eta(dH) - (\mathcal{R}(H) + H)\mathcal{R}.$$
\end{Def}
Locally, it has the local expression 
$$X_H = \partder{H}{p_i} \partder{}{q^i} - \left (\partder{H}{q^i} + p_i\partder{H}{z} \right) \partder{}{p_i} + \left( p_i \partder{H}{p_i} - H\right ) \partder{}{z}.$$ The dynamics corresponding to a Hamiltonian vector field $X_H$ are determined by
\begin{align*}
\totder{q^i}{t} &= \partder{H}{p_i},\\
\totder{p_i}{t} &= - \partder{H}{q^i} - p_i \partder{H}{z},\\
\totder{z}{t} &= p_i \partder{H}{p_i} - H.
\end{align*}
For instance, if we take as Hamiltonian $$H := \frac{p^2}{2m} + \frac{\kappa^2 m q^2}{2} + \gamma z,$$ the equations determined by $X_H$ are 
\begin{align*}
    \totder{q}{t} &= \frac{p}{m},\\
    \totder{p}{t} &=  - \kappa^2 m q - \gamma p,\\
    \totder{z}{t} &= \frac{p^2}{2m} - \frac{\kappa^2 m q^2}{2} - \gamma z,
\end{align*}
which are precisely the equations for the damped harmonic oscillator.\\

We can define a symplectic structure in $TM$ taking $\Omega_0 := \flat_\eta^* \Omega_M$, where $\Omega_M$ is the canonical symplectic structure in $T^*M$. In local coordinates $(q^i, p_i, z, \dot{q}^i, \dot{p}_i, \dot{z})$, it has the expression
\begin{align*}
    \Omega_0 =\,& p_ip_j dq^i \wedge d\dot{q}^j + \left((1 + \delta_i ^j)p_i \dot{q}^j  - \delta_i^j\dot{z}\right) dq^i \wedge dp_j - dq^i \wedge d\dot{p}_i - p_i dq^i \wedge d\dot{z}\\
     &+ dp_i \wedge d\dot{q}^i + dz \wedge d\dot{z} - p_i dz \wedge d\dot{q}^i - \dot{q}^i dz \wedge dp_i \\
     = \,& dq^i \wedge \left (p_ip_i dq^j + (1 + \delta_i^j)p_i \dot{q}^j dp_j - \dot{z} dp_i - d\dot{p}_i - p_i d\dot{z}\right) \\
     & + dp_i \wedge d\dot{q}^i  + dz \wedge \left( d\dot{z} - p_id\dot{q}^i - \dot{q}^i dp_i\right)
\end{align*}
 \begin{Def}[Gradient vector field] Given a Hamiltonian $H$ on $M$, define the \textbf{gradient vector field} of $H$ as $$\grad H := \sharp_\eta(dH).$$
 \end{Def}
 Locally, it is given by $$\grad H = \partder{H}{p_i} \partder{}{q^i} - \left (\partder{H}{q^i} + p_i\partder{H}{z} \right) \partder{}{p_i} + \left( p_i \partder{H}{p_i} + \partder{H}{z}\right ) \partder{}{z}.$$

 We have the following relation between both vector fields $$X_H = \grad H - (\mathcal{R}(H) + H ) \mathcal{R}.$$
 
 Just like in the previous sections, a vector field $X: M \rightarrow TM$ is locally a gradient vector field if and only if it defines a Lagrangian submanifold in $(TM , \Omega_0)$. The proof is straight-forward, checking that $$X^*\Omega_0 = - dX^{\flat_\eta}.$$
We can also interpret the Hamiltonian vector field $X_H$ as a Lagrangian submanifold of $TM$, but we need to modify slightly the symplectic form. It is easy to verify that $$X_H^*\Omega_0 = \d(\mathcal{R}(H) \eta + H\eta),$$ therefore, taking $$ \Omega_H:= \Omega_0 - \d(\mathcal{R}(H) \eta + H\eta)^v,$$ we have $$X_H^*\Omega_H = 0.$$ It is clear that $\Omega_H$ is a symplectic form. We have proved:\\

\begin{prop} The Hamiltonian vector field $X_H: M \rightarrow TM$ defines a Lagrangian submanifold of the symplectic manifold $(TM,  \Omega_H)$.
\end{prop}

Now we study evolution vector fields, an important vector field in the application of contact geometry to thermodynamics.\\

\begin{Def} Given a Hamiltonian $H$, we define the \textbf{evolution vector field} as $$\mathcal{E}_H := X_H + H \mathcal{R}.$$
\end{Def}

Locally, the evolution vector field is written $$\mathcal{E}_H = \partder{H}{p_i} \partder{}{q^i} - \left (\partder{H}{q^i} + p_i\partder{H}{z} \right) \partder{}{p_i} + p_i \partder{H}{p_i}\partder{}{z}.$$

Let us see how we can modify the symplectic form $\Omega_0$ in such a way that $\mathcal{E}_H$ defines a Lagrangian submanifold. We have $$\mathcal{E}_H^*\Omega_0 = X_H^* \Omega_0 + (H \mathcal{R}) ^* \Omega_0 = \d(\mathcal{R}(H) \eta + H\eta) - \d(H \eta) = \d( \mathcal{R}(H) \eta),$$ and thus, $\mathcal{E}_H$ defines a Lagrangian submanifold of $(TM , \widetilde \Omega_H),$ where $$\widetilde \Omega_H = \Omega_0 - \d(\mathcal{R}(H) \eta).$$

We can also interpret Hamiltonian and evolution vector fields as Legendrian submanifolds of a certain contact structure defined on $TM \times \mathbb{R}.$\\
 
\begin{Def} Let $(M, \eta)$ be a contact manifold. Define the contact form on $TM \times \mathbb{R}$ as $$\hat{\eta} := \eta^c + t \eta^v,$$  where $\eta^c$, $\eta^v$ are the complete and vertical lifts \cites{de2011methods}. It is easily checked that $\hat{\eta}$ defines a contact structure \cites{de2019contact}. 
\end{Def}
In local coordinates it has the expression: 

$$\hat{\eta} = d\dot{z} - \dot{p}_i dq^i - p_i d\dot{q}^i + t dz - tp_i dq^i$$

We have the following \cites{de2019contact}:\\
\begin{prop} Let $X_H: M \rightarrow TM$ be a Hamiltonian vector field. Then, the submanifold defined by the immersion $$i: M \hookrightarrow TM \times \mathbb{R};\,\,\,\,p \mapsto (X_H(p), \mathcal{R}(H))$$ is a Legendrian submanifold of $(TM\times \mathbb{R}, \hat{\eta})$.
\end{prop}
\begin{proof} Using the properties of complete and vertical lifts we have $$(X_H \times \mathcal{R}(H)) ^* \hat{\eta} = \mathcal{L}_{X_H}\eta + \mathcal{R}(H) \eta.$$ Using \textcolor{red}{Lemma \ref{isotropicClassificationContact}} it will be sufficient to see that $\mathcal{L}_{X_H} \eta = - \mathcal{R}(H).$ This is a straight-forward verification using $$\mathcal{L}_{X_H} \eta = di_{X_H} \eta + i_{X_H} d\eta = - dH + dH - \mathcal{R}(H) \eta. $$
\end{proof}
\subsection{An overview of the use of contact geometry in thermodynamics}
A thermodynamical system is characterized by the following variables $$(H, T, S, P_i, V^i),$$ corresponding to energy, temperature, entropy, and 
the generalized pressures and volumes, respectively. We will denote $q^i = V^i$, following the notation used so far. The energy of the system is a function 
\begin{align*}
    H : & \; T^\ast Q \times \mathbb{R} \rightarrow \mathbb{R}; \,\,(q^i, p_i, S) \mapsto H(q^i,p_i, S).
\end{align*}
The first law of thermodynamics can be written as $$d H = \delta \mathcal{Q} - \delta \mathcal{W},$$ where $\delta \mathcal{Q}$ and $\delta \mathcal{W}$ are one forms representing the heat and work, respectively. We will assume that these forms have local expressions 
$$\delta \mathcal{Q} = T dS, \,\, \delta \mathcal{W} = P_i dq^i,$$ 
for some functions $P_i, T$ that physically represent the conjugate variables to $q^i$ (pressure, if $q^i$ represented volume), and $S$ (temperature, if $S$ was the entropy), respectively. Then, the first law of thermodynamics reads $$d H = Td S  - P_i dq^i.$$ If we were studying an isolated system, energy would be conserved, that is $$0 = Td S  - P_i dq^i$$ or, dividing by $T$, $$0 = d S - \frac{P_i}{T} d q^i.$$ Identifying $p_i = P_i/T$, after a change of variables if necessary, we obtain $$0 = dS - p_i dq^i.$$ Defining $$\eta := d S - p_i dq^i,$$ the canonical contact form in $T^\ast Q \times \mathbb{R},$ we conclude that isolated processes take place in Legendrian submanifolds of $T^\ast Q \times \mathbb{R},$ motivating their study.\\

It turns out that the integral curves of the evolution vector field $\mathcal{E}_H$ can be interpreted as an isolated system\\

\begin{prop} The integral curves of $\mathcal{E}_H$ satisfy $$\totder{H}{t} = 0.$$ Furthermore, locally $$\totder{S}{t} = p_i dq ^i = \frac{P_i}{T} d q^i.$$
\end{prop}
\begin{proof} Indeed, by definition we know that $$i_{\grad H} d \eta + \eta(H) \eta = d H.$$ Therefore,
\begin{align*}
    \totder{H}{t} &= d H \cdot  \mathcal{E}_H = d H \cdot X_H + H \mathcal{R}(H) = d H \cdot \grad H - (\mathcal{R}(H) + H) \mathcal{R}(H) + H \mathcal{R}(H)\\
    &=(i_{\grad H} d \eta + \eta(H) \eta)\cdot \grad H - (\mathcal{R}(H))^2  = 0.
\end{align*}
Now, the last part is a consequence of the equality $$d H = d S - p_i dq^i.$$
\end{proof}
For more details on the subject, we refer to \cites{ArnoldThermo, GeometrySomeThermo,simoes2020contact,simoes2021ThermodynamicSystems, MRUGALAcontactstrcuture, MRUGALAcontinuouscontact}, (see also two modern approaches \cites{bravetti1,bravetti2}).

\subsection{Coisotropic reduction}

Coisotropic reduction in contact manifolds has been developed in \cites{de2019contact} (see also \cites{Tortorella2017,le2019deformations}).\\

The following definition will result useful. Given a subspace $\Delta_q \subseteq T_qM$, we define the $d\eta$-\textbf{orthogonal complement} as $$\Delta_q^{\perp_{d\eta}}:= \{v \in T_qM,\,|\,\, d\eta(v,w) = 0\,\, \forall w \in \Delta_q\}.$$
\begin{prop} \label{ContactOrthogonalGeneral} Let $\Delta_q \subseteq T_qM$ be a subspace. Then $$\Delta_q^{\perp_{d\eta}} \cap \mathcal{H}_q \subseteq \Delta_q^{\perp_\Lambda}.$$ Furthermore, if $\mathcal{R}(q) \in \Delta_q$ or $\Delta_q \subseteq \mathcal{H}_q$, the equality holds.
\end{prop}
\begin{proof} Let $v \in \Delta_q^{\perp_{d\eta}}\cap \mathcal{H}_q$ and take $\alpha := i_v d \eta$. It is clear that $\alpha \in \Delta_q^0$. We will prove that $\sharp_\Lambda(-\alpha) = v$. Indeed, for $\beta \in T^*_qM$, since $b_\eta(v) = \alpha$ (as a direct calculation shows), we have
\begin{align*}
\langle \beta, \sharp_\Lambda(\alpha) \rangle &= \Lambda(\alpha, \beta) = \Omega(\sharp_\eta(\alpha), \sharp_\eta(\beta)) = \Omega(v, \sharp_\eta(\beta)) \\
&= -\Omega( \sharp_\eta(\beta),v) - \eta(\sharp_\Lambda(\beta))\eta(v) =- \langle \flat_\eta(\sharp_\eta(\beta)), v\rangle = -\langle \beta, v \rangle,
\end{align*}
that is, $v = \sharp_\Lambda(-\alpha)$.\\

Now, if $\mathcal{R}(q) \in \Delta_q$, we compare dimensions. Since $\ker \sharp_\Lambda = \langle \eta\rangle$ and $\eta \not \in \Delta_q^0$, we have $$\dim \Delta_q^{\perp_\Lambda} = \dim \Delta_q^0 = \dim M - \dim \Delta_q.$$
Furthermore, $\Delta^{\perp_{d\eta}}\cap \mathcal{H}_q$ has the same dimension, since
\begin{align*} 
\Delta_q^{\perp_{d\eta}} \cap \mathcal{H}_q &= (\Delta_q \cap \mathcal{H}_q \oplus \mathcal{V}_q)^{\perp_{d\eta}} \cap \mathcal{H}_q = (\Delta_q \cap \mathcal{H}_q)^{\perp_{d\eta}} \cap \mathcal{V}_q^{\perp_{d\eta}} \cap \mathcal{H}_q\\ & = (\Delta_q \cap \mathcal{H}_q)^{\perp_{d\eta}} \cap \mathcal{H}_q.
\end{align*}
This latter is just the symplectic complement in $(\mathcal{H}_q, d\eta |_{\mathcal{H}})$ and hence, $$\dim (\Delta_q^{\perp_{d\eta}}\cap \mathcal{H}_q) = \dim \mathcal{H}_q - \dim (\Delta_q \cap \mathcal{H}_q) = \dim M - 1 - (\dim \Delta_q - 1).$$

Now, if $\Delta_q \subseteq \mathcal{H}_q$, $\eta \in \Delta_q^0$ and, thus, $$\dim \Delta_q^{\perp_{d\eta}} = \dim M - \dim \Delta_q - 1.$$
Since $\Delta_q^{\perp_{d\eta}} \cap \mathcal{H}_q$ is just the symplectic complement of $\Delta_q$ we have $$\dim (\Delta_q^{\perp_{d\eta}} \cap \mathcal{H}_q) = \dim \mathcal{H}_q - \dim \Delta_q = \dim M - 1 - \dim \Delta_q.$$
\end{proof}
This proposition allows us to characterize Legendrian submanifolds:\\
\begin{lemma} If $L\hookrightarrow M$ is a Legendrian submanifold, then $L$ is horizontal and $\dim L = n$ (where $\dim M = 2n+1$). Furthermore, if $L$ is horizontal and isotropic (or coisotropic) with $\dim L = n$, $L$ is Legendrian.
\end{lemma}
\begin{proof} Since $\sharp_\Lambda$ takes values in $\mathcal{H}$, it is clear that every Legendrian submanifold is horizontal. Since $L$ is horizontal, $$\dim T_qL^{\perp_\Lambda} = \dim M - \dim L - 1.$$ From the previous equation and using that $T_qL^{\perp_\Lambda} = T_qL$, we deduce that $\dim L = \dim M - \dim L - 1$. This implies $\dim L = n$. The last property is easily seen via a direct comparison of dimensions.
\end{proof}
We will also need characterization of isotropic submanifolds in contact geometry:
\begin{lemma} \label{isotropicClassificationContact} A submanifold $N \hookrightarrow M$ is isotropic if and only if $i^* \eta = 0$.
\end{lemma}
\begin{proof}Necessity is clear, since $\sharp_\Lambda$ takes values in $\mathcal{H}$. Now suppose that $N$ is horizontal. We have $i^*\eta = 0$ and thus, $i^* d\eta = 0$. This implies that $T_qN \subseteq T_qN^{\perp_{d\eta}} \cap \mathcal{H}_q \subseteq T_qN^{\perp_\Lambda},$ using \textcolor{red}{Proposition \ref{ContactOrthogonalGeneral}}.
\end{proof}	

\begin{prop} \label{ContactOrthogonalCoisotropic} Let $i:N \hookrightarrow M$ be a coisotropic submanifold such that $\mathcal{R}(q) \in T_qN$ for every $q\in N$ or $T_qN \subseteq \mathcal{H}_q$ for every $q \in N$. Define $\eta_0:= i^*\eta$. Then $$T_qN^{\perp_\Lambda} = \ker{ d\eta_0} \cap \ker \eta_0.$$
\end{prop}
\begin{proof} Let $q\in N$. \textcolor{red}{Proposition \ref{ContactOrthogonalGeneral}} implies that $$T_qN^{\perp_\Lambda} = T_qN^{\perp_\eta} \cap \mathcal{H}_q.$$ But, since $T_qN$ is coisotropic, then it is just $\ker d\eta_0 \cap \ker \eta_0.$ 
\end{proof}

We then have the following result:\\
\begin{prop} Let $i: N \hookrightarrow M$ be a coisotropic subamnifold such that $\mathcal{R}(q) \in T_qN$ for every $q \in N$ or $T_qN \subseteq \mathcal{H}_q$ for every $q \in N$. Then, the distribution $TN^{\perp_\Lambda}$ defined by $q \mapsto T_qN^{\perp_\Lambda}$ is involutive.
 \end{prop}
\begin{proof} Denote $\eta_0:= i^*\eta$ and let $X, Y$ be vector fields along $N$ taking values in $TN^{\perp_\Lambda}$. \textcolor{red}{Proposition \ref{ContactOrthogonalCoisotropic}} implies that $$i_X  d\eta_0 = i_Y d\eta_0 = 0;\,\,\,\, i_X \eta_0 = i_Y \eta_0 = 0.$$
It suffices to check that $$i_{[X,Y]} d\eta_0 = 0;\,\,\,\, i_{[X,Y]}\eta_0 =0.$$ Indeed, taking $Z$ an arbitrary vector field in $N$, we have
\begin{align*}
0 &= \d^2\eta_0(X,Y,Z) =  X( d\eta_0(Y,Z)) - Y( d\eta_0(X,Z)) + Z( d\eta_0(X,Y))\\
&-  d\eta_0([X,Y],Z) +  d\eta_0([X,Z],Y) - d\eta_0([Y,Z], X) = -d\eta_0([X,Y],Z),
\end{align*}
where we have used that $X,Y \in \ker d \eta_0.$ In a similar way we obtain 
\begin{align*}
0 = d\eta_0(X,Y) = X(\eta_0(Y)) - Y(\eta_0(X)) - \eta_0([X,Y]) = -\eta_0([X,Y]),
\end{align*}
that is, $[X,Y] \in \ker  d\eta_0 \cap \ker \eta_0 = TN^{\perp_\Lambda}.$
\end{proof}

\subsection{Vertical coisotropic reduction}
We will restrict the study to vertical submanifolds, that is, submanifolds satisfying $\mathcal{R}(q) \in T_qN$, for every $q \in N$. Notice that if $N$ is a coisotropic vertical submanifold, the distribution $(TN)^{\perp_\Lambda}$ is regular of rank $$\operatorname{rank} \, (TN)^{\perp_\Lambda} = \dim M - \dim N.$$
\begin{theorem}[Vertical coisotropic reduction in the contact setting] \label{VerticalReductionContact} Let $(M, \eta)$ be a contact manifold and $i :N \hookrightarrow M$ be a coisotropic submanifold such that $\mathcal{R}(q) \in T_qN$ for every $q \in N$. If the space of all leaves $N /\mathcal{F}$ admits a manifold structure such that the projection $\pi: N \rightarrow N/\mathcal{F}$ is a submersion, then there exists a unique $1$-form $\eta_N$ in $N/\mathcal{F}$ such that $\pi^*\eta_N = i_*\eta $ and $(N/\mathcal{F},\eta_N)$ is a contact manifold.
\end{theorem}
\begin{proof}
Denote $\eta_0:= i^*\eta$. Uniqueness is clear from the imposed relation since it forces us to define $$\eta_N([u]):= \eta_0(u).$$ It only remains to check well-definedness and that it defines a contact manifold. That this definition does not depend on the chosen representative vector is clear since a vector tangent to the distribution is necessarily in the kernel of $\eta$. Furthermore, if $X$ is a vector field tangent to the distribution $TN^{\perp_\Lambda}$, \textcolor{red}{Proposition \ref{ContactOrthogonalCoisotropic}} implies $$\mathcal{L}_X \eta_0 = 0 = di_X \eta_0 + i_Xd\eta_0 = 0,$$ since $X \in \ker d\eta_0 \cap \ker \eta_0$. \\

  To check that it is a contact manifold, we calculate the dimension of $N /\mathcal{F}$. We have that $\operatorname{rank} \, (T_qN)^{\perp_\Lambda} = \dim M - \dim N$. We conclude, taking $k:= \dim N$, that $$\dim N/\mathcal{F} = 2 \dim N - \dim M = 2(k-n-1) + 1$$ and therefore, $(N/\mathcal{F}, \eta_N)$ is a contact manifold if and only if $$\eta_N \wedge ( d\eta_N)^{k-n-1}\neq 0.$$ Since $\pi$ is a submersion, this is equivalent to $\eta_0 \wedge ( d\eta_0)^{k-n-1}\neq0$. This is straightforward using \textcolor{red}{Proposition \ref{ContactOrthogonalCoisotropic}}.
\end{proof}

\subsubsection{Projection of Legendrian submanifolds}
Now we check that the image of a Legendrian submanifold $L \hookrightarrow M$ under the projection $\pi: N \rightarrow N/\mathcal{F}$ is again a Legendrian submanifold.\\

\begin{prop} \label{ProjectionLegendrianVerticalContact}Let $L \hookrightarrow M$ be a Legendrian submanifold such that $L$ and $N$ have clean intersection. If $ L_N:=  \pi(L \cap N)$ is a submanifold of $N/ \mathcal{F}$, then $L_N$ is Legendrian.
\end{prop}
\begin{proof} It suffices to check that $L_N$ is horizontal, isotropic and $\dim L_N = \dim N - n - 1$ using \textcolor{red}{Lemma \ref{isotropicClassificationContact}}.  Since $L$ is horizontal, $L_N$ is horizontal and thus, $L_N$ is isotropic.\\

Comparing dimensions, we have
\begin{equation}\label{eq1contact}
\dim T_{[q]}L_N = \dim T_qL\cap T_qN - \dim T_qL \cap T_qN^{\perp_\Lambda}.
\end{equation}
Now, since $(T_qL \cap T_qN^{\perp_\Lambda})^{\perp_\Lambda} = T_qL + (T_qN^{\perp_\Lambda})^{\perp_\Lambda}$ and $(T_qL \cap T_qN)^{\perp_\Lambda}$ is horizontal, we have 
\begin{equation}\label{eq2contact}
\dim( T_qL + (T_qN^{\perp_\Lambda})^{\perp_\Lambda}) = \dim M - \dim(T_qL \cap T_qN^{\perp_\Lambda}) - 1.
\end{equation}
Using $\dim(T_qL \cap T_qN^{\perp_\Lambda}) = \dim(T_qL + T_qN) - 1$ and substituting (\ref{eq2contact})) in (\ref{eq1contact}), we obtain
\begin{align*}
\dim L_N& = \dim L \cap N - \left( \dim M - \dim(T_qL + T_qN )\right) \\
&= \dim L \cap N - \dim M + \dim N + \dim L - \dim L \cap N \\
&= \dim N - 2n - 1 + n = \dim N - n -1.
\end{align*}
\end{proof}

\subsection{Horizontal coisotropic reduction}
We will restrict the study to \textbf{horizontal} coisotropic submanifolds $N \rightarrow M$, that is, manifolds satisfying $T_qN \subseteq \mathcal{H}_q$ for every $q\in N$. \\
\begin{remark}{\rm  Notice that in this case reduction is trivial, since the only coisotropic horizontal submanifolds of a contact manifold are those that are Legendrian. This would imply $$\dim N/\mathcal{F} = 0,$$ making the resulting manifold trivial.}
\end{remark}



Given an arbitrary coisotropic submanifold $N \hookrightarrow M$, we cannot guarantee the well-definedness of the $2$-form in the quotient $N/\mathcal{F}$ (actually, in the contact setting, we cannot even guarantee the integrability of $TN^{\perp_\Lambda}$) so this time (referring to horizontal reduction in cosymplectic geometry) we cannot obtain a foliation of $N$ in symplectic leaves, since $TN \cap \mathcal{H}|N$ is not integrable in the general setting. \\

\begin{remark}{\rm The triviality of this case makes the projection of Lagrangian submanifolds trivial.}
\end{remark}



\section{An interlude: Different geometries provide different dynamics}\label{interlude}

In this section we will show how different geometric structures on the same phase space can provide different dynamics for the same Lagrangian or Hamiltonian function. This fact explains the different equations of motion between the cosymplectic case (time-dependent Lagrangians) and the contact case (action-dependent Lagrangians). At the end of the section we will also see that Hamilton's principle can be generalised to the so-called Herglotz principle, which gives a new way to obtain these different dynamics.

\subsection{The Lagrangian picture}

In order to distinguish the different dynamics discussed in this paper and their physical nature, it is useful to recall their Lagrangian formulation.

Assume that $L : TQ  \longrightarrow \mathbb R$ is a Lagrangian function, 
where $Q$ is a $n$-dimensional configuration manifold. Then, $L = L(q^i, \dot{q}^i)$, where
$(q^i)$ are coordinates in $Q$ and $(q^i, \dot{q}^i)$ are the induced bundle coordinates in $TQ$ (positions and velocities).
We will assume that $L$ is regular, that is, the Hessian matrix
$$
\left( \frac{\partial^2 L}{\partial \dot{q}^i \partial \dot{q}^j} \right)
$$
is regular.
Using the canonical endomorphism $S$ on $TQ$ locally defined by
$$
S = d q^i \otimes \frac{\partial}{\partial \dot{q}^i}
$$
one can construct a 1-form $\lambda_L$ defined by
$$
\lambda_L = S^* (dL)
$$
and the 2-form
$$
\omega_L = - d\lambda_L
$$
Then, $\omega_L$ is symplectic if and only if $L$ is regular.

In that case, we have the corresponding vector bundle isomorphism
\begin{eqnarray*}
&&\flat_{\omega_L} :  T(TQ) \longrightarrow T^*(TQ)\\
&&\flat_{\omega_L} (v) = i_v \, \omega_L 
\end{eqnarray*}
and the Hamiltonian vector field
$$
\xi_L = X_{E_L},
$$
defined by 
$$
\flat_{\omega_L}(\xi_L) = dE_L,
$$
where $E_L = \Delta(L) -L$ is the energy, and $\Delta = \dot{q}^i \, \frac{\partial}{\partial \dot{q}^i}$ is the Liouville vector field on $TQ$.\\

The vector field $\xi_L$ (the Euler-Lagrange vector field) is a SODE (second order differential equation), that is, its integral curves are just the tangent lifts of its projections to $Q$. These projections are called the solutions of $\xi_L$, and satisfy the usual Euler-Lagrange equations
\begin{equation}
\frac{d}{dt} \left(\frac{\partial L}{\partial \dot{q}^i}\right) - \frac{\partial L}{\partial q^i} = 0.
\end{equation}

Next, we  recall here the geometric formalism for time-dependent Lagrangian systems.
In this case, we also have a regular Lagrangian $L : TQ \times \mathbb R \longrightarrow \mathbb R$,
and we consider the cosymplectic structure given by the pair
$(\Omega_L, dz)$ (in this case $z$ represents \textit{time}), where
$$
\Omega_L = - d \lambda_L
$$
It is esay to check that $L$ is regular if and only if
$$
dz \wedge \Omega_L^n \not= 0.
$$
In that case, we have a cosymplectic structure and, defining $(W^{ij})$ to be the inverse matrix of $$(W_{ij}) = \left( \frac{\partial^2 L}{\partial \dot{q}^i \partial \dot{q}^j} \right),$$ the Reeb vector field of the cosymplectic manifold $(TQ, \Omega_L, dz)$ is locally given by
$$
\mathcal R = \frac{\partial}{\partial z} - W^{ij} \frac{\partial^2 L}{\partial \dot{q}^j \partial z} \, \frac{\partial}{\partial \dot{q}^i}.
$$

Recall that we had the vector bundle isomorphism
\begin{eqnarray*}
&&{\flat}_{dz, \Omega_L}  :  T(TQ \times \mathbb R) \longrightarrow T^*(TQ \times \mathbb R)\\
&&{\flat}_{dz, \Omega_L} (v) = i_v \, \Omega_L + dz (v) \, dz
\end{eqnarray*}
which gives the vector fields defined in \textcolor{red}{Section \ref{Cosymplectic}}. In particular, we have the evolution vector field $$\mathcal{E}_{E_L} = X_{E_L} + \mathcal{R},$$ where $E_L$ is the energy of the system, $$E_L = \Delta(L) - L = \dot{q}^i \partder{L}{q^i} - L.$$

Now, if $(q^i(t), \dot{q}^i(t), z(t))$ is an integral curve of ${\mathcal E}_L$, then its projection to $Q$ satisfies the usual Euler-Lagrange equations
\begin{equation}\label{cosylagr3}
\frac{d}{dt} \left(\frac{\partial L}{\partial \dot{q}^i}\right) - \frac{\partial L}{\partial q^i} = 0,
\end{equation}
since $z = t + constant$. 

Finally, let $L : TQ \times \mathbb R \longrightarrow \mathbb R$ be a Lagrangian function, $L = L(q^i, \dot{q}^i, z)$, where $z$ is a global coordinate on $\mathbb{R}$, representing \textit{action}.

We will assume that $L$ is regular, that is, the Hessian matrix
$$
\left( \frac{\partial^2 L}{\partial \dot{q}^i \partial \dot{q}^j} \right)
$$
is regular. So, we construct a 1-form $\lambda_L$ defined by
$$
\lambda_L = S^* (dL)
$$
where now $S$ and $S^*$ are the natural extension of $S$.

Now, the 1-form
$$
\eta_L = dz -  \frac{\partial L}{\partial \dot{q}^i} \, dq^i
$$
is a contact form on $TQ \times \mathbb R$ if and only if $L$ is regular, and then
$$
\eta_L \wedge (d\eta_L)^n \not= 0.
$$ 
The corresponding Reeb vector field is
$$
{\mathcal R} = \frac{\partial}{\partial z} - W^{ij} \frac{\partial^2 L}{\partial \dot{q}^j \partial z} \, \frac{\partial}{\partial \dot{q}^i}.
$$
The energy of the system is defined as in the precedent cases by 
$$
E_L = \Delta (L) - L
$$
where $\Delta = \dot{q}^i \, \frac{\partial}{\partial \dot{q}^i}$ is the Liouville vector field on $TQ$ extended in the usual way to $TQ \times \mathbb R$. The vector bundle isomorphism of contact manifolds defined in \textcolor{red}{Section \ref{Contact}} is given by
$$
{\flat_{\eta_L}} : T(TQ \times \mathbb R) \longrightarrow T^* (TQ \times \mathbb R);\,\,\,
{\flat_{\eta_L}} (v) = i_v (d\eta_L) + (i_v \eta_L) \cdot \eta_L.
$$
We shall denote its inverse by ${\sharp_{\eta_L}} = ({\flat_{\eta_L})}^{-1}$.

Denote by ${\xi}_L$ the unique vector field defined by the equation
\begin{equation}\label{clagrangian1}
{\flat_{\eta_L}} ({\xi}_L) = dE_L - (\mathcal R(E_L) + E_L) \, \eta_L
\end{equation}
Note that this is precisely the Hamiltonian vector field of $E_L$ defined in \textcolor{red}{Section \ref{Contact}}.
A direct computation from eq. (\ref{clagrangian1}) shows that,
if $(q^i(t), \dot{q}^i(t), z(t))$ is an integral curve of $\bar{\xi}_L$, we obtain
$$
{\ddot{q}}^i \, \frac{\partial}{\partial \dot{q}^i}\left(\frac{\partial L}{\partial \dot{q}^j}\right) 
+ \dot{q}^i \, \frac{\partial}{\partial q^i}\left(\frac{\partial L}{\partial \dot{q}^j}\right) 
+  \dot{z} \frac{\partial}{\partial z}\left(\frac{\partial L}{\partial \dot{q}^i}\right) - \frac{\partial L}{\partial q^i} =
\frac{\partial L}{\partial \dot{q}^i} \frac{\partial L}{\partial z}
$$
with an additional equation $\dot{z} = L$. These equations
correspond to the generalized Euler-Lagrange equations considered by G. Herglotz in 1930.
\begin{equation}\label{clagrangian4}
\frac{d}{dt} \left(\frac{\partial L}{\partial \dot{q}^i}\right) - \frac{\partial L}{\partial q^i} =
\frac{\partial L}{\partial \dot{q}^i} \frac{\partial L}{\partial z}
\end{equation}
Notice that Herglotz equations depend on the action, so this type of Lagrangians are called in physics \textit{dependent on the action}.

\subsection{The Herglotz principle}

In order to give additional differences between the usual Hamilton principle and Herglotz principle, it is interesting to recall briefly the last one and how it is a natural generalization of the former one.

Let $L:TQ \times \mathbb{R} \to \mathbb{R}$ be a Lagrangian function. 

Fix $q_1,q_2 \in Q$ and an interval $[a,b] \subset \mathbb{R}$. We denote by $\Omega(q_1,q_2, [a,b]) \subseteq({\cal C}^\infty([a,b]\to Q))$ the space of smooth curves $\xi$ such that $\xi(a)=q_1$ and $\xi(b)=q_2$. This space has the structure of an infinite dimensional smooth manifold whose tangent space at $\xi$ is given by the set of vector fields over $\xi$ that vanish at the endpoints, that is,
\begin{equation}
\begin{aligned}
        T_\xi \Omega(q_1,q_2, [a,b]) =  \{&
            v_\xi \in {\cal C}^\infty([a,b] \to TQ) \mid \\& 
            \tau_Q \circ v_\xi = \xi, \,v_\xi(a)=0, \, v_\xi(b)=0 
            \}.
\end{aligned}
\end{equation}

We will consider the following maps. Fix an initial action $c \in \mathbb{R}$. Let 
\begin{equation}
    \mathcal{Z}:\Omega(q_1,q_2, [a,b])  \to {\cal C}^\infty ([a,b] \to \mathbb{R})
\end{equation}
 be the operator that assigns to each curve $\xi$ the curve $\mathcal{Z}(\xi)$ that solves the following ODE:
\begin{equation}\label{contact_var_ode}
    \frac{d \mathcal{Z}(\xi)(t)}{dt} = L(\xi(t), \dot \xi(t), \mathcal{Z}(\xi)(t)), \quad \mathcal{Z}(\xi)(a)= c.
\end{equation}

Now we define the \emph{action functional} as the map which assigns to each curve the solution to the previous ODE evaluated at the endpoint:
\begin{equation}\label{contact_action}
    \begin{aligned}
        \mathcal{A}: \Omega(q_1,q_2, [a,b]) &\to \mathbb{R},\\
        \xi &\mapsto \mathcal{Z}(\xi)(b),
    \end{aligned}
\end{equation}
that is, $\mathcal{A} = ev_b \circ \mathcal{Z}$, where $ev_b: \zeta \mapsto \zeta(b)$ is the evaluation map at $b$. We have\\

\begin{theorem}
    Let $L: TQ \times \mathbb{R} \to \mathbb{R}$ be a Lagrangian function and let $\xi\in  \Omega(q_1,q_2, [a,b])$ be a curve in $Q$. Then, $(\xi,\dot\xi, \mathcal{Z}(\xi))$ satisfies the Herglotz's equations  if and only if $\xi$ is a critical point of $\mathcal{A}$.
\end{theorem}
\bigskip
\begin{remark}{\rm
    This theorem generalizes Hamilton's Variational Principle. In the case that the Lagrangian is independent of the $\mathbb{R}$ coordinate (i.e., $L(q^i,\dot{q}^i,z)=\hat{L}(q^i, \dot{q}^i)$)  
the contact Lagrange equations reduce to the usual Euler-Lagrange equations. In this situation, we can integrate the ODE of (\ref{contact_action}) and we get
    \begin{equation}
    \mathcal{A}(\xi) = \int_a^b \hat L(\xi(t),\dot\xi(t))d t + \frac{c}{b-a},
    \end{equation}
    that is, the usual Euler-Lagrange action up to a constant.}
\end{remark}

\subsection{The Legendre transformation and the Hamiltonian picture}

Given a Lagrangian function $L : TQ \times \mathbb R \longrightarrow \mathbb R$ we can define the Legendre transformation 
$$
\mathbb{F}L : TQ \times \mathbb R \longrightarrow T^*Q \times \mathbb R
$$
given by
$$
\mathbb{F}L (q^i, \dot{q}^i, z) = (q^i, \hat{p}_i, z)
$$
where
$$
\hat{p}_ i = \frac{\partial L}{\partial \dot{q}^i}
$$
A direct computation shows that
$$
\mathbb{F}L ^* \lambda_Q = \lambda_L,
$$
where $\lambda_Q$ is the canonical Liouville form on $T^*Q$, here extended to the product manifold $ T^*Q \times \mathbb R$. If the Lagrangian does not depend on time, then the Legendre transformation reduces to a bundle morphism
$$
\mathbb{F}L : TQ \longrightarrow T^*Q 
$$
and then the pull-back of the canonical symplectic form $\omega_Q$ on $T^*Q$ is just $\omega_L$. In addition, the energy $E_L$ corresponds via the Legendre transformation to the Hamiltonian energy $H$ such that
$$
(\mathbb{F}L)^*(H) = E_L.
$$
Since the corresponding geometric structures (symplectic, cosymplectic and contact) on both sides are preserved by the Legendre transformation, one concludes that corresponding dynamics are $\mathbb{F}L$-related (see \cites{canaria} for the details).

\section{Coisotropic reduction in cocontact geometry}\label{Cocontact}
\setsectiontitle{COISOTROPIC REDUCTION IN COCONTACT GEOMETRY}

Cocontact manifolds have been introduced in \cites{de2022time} just to provide a setting for dissipative systems which also depend on time. In geometric terms, we are combining cosymplectic and contact structures.\\

\begin{Def}[Cocontact manifold] A cocontact manifold is a triple $(M,\theta, \eta)$, where $M$ is a $(2n + 2)$-dimensional manifold, $\theta$ is a closed $1$-form, $\eta$ is a $1$-form  and, $\theta \wedge \eta \wedge (d\eta)^n \neq 0$ is a volume form.
\end{Def}

The bundle isomorphism in this case is defined as $$ \flat_{\theta, \eta}: TM \rightarrow T^*M;\,\, v \mapsto \theta(v)\theta + i_{v}d\eta + \eta(v) \eta,$$ and its inverse is denoted by $\sharp_{\theta,\eta} = \flat_{\theta,\eta}^{-1}.$\\

In cocontact geometry there exists as well a set of canonical coordinates $(q^i,p_i,z,t)$, which will be called Darboux coordinates, such that $$\eta = dz - p_i dq^i; \,\, \theta = dt.$$ We can define as well the Reeb vector fields as $$\mathcal{R}_z:= \sharp_{\theta, \eta}(\eta); \,\, \mathcal{R}_t = \sharp_{\theta, \eta}(\theta),$$ which can be expressed locally as $$\mathcal{R}_z = \partder{}{z}; \,\, \mathcal{R}_t = \partder{}{t}.$$ We also have vertical and horizontal distributions:
\begin{itemize}
    \item[$i)$] The \textbf{$z$-horizontal} distribution, $\mathcal{H}_z := \ker \eta$;
    \item[$ii)$] The \textbf{$t$-horizontal} distribution, $\mathcal{H}_t := \ker \theta$;
    \item[$iii)$] The \textbf{$tz$-horizontal} distribution $\mathcal{H}_{tz} := \mathcal{H}_t \cap \mathcal{H}_z$;
    \item[$iv)$] The \textbf{$t$-vertical} distribution, $\mathcal{V}_t := \langle \mathcal{R}_t \rangle$;
    \item[$v)$] The \textbf{$z$-vertical} distribution, $\mathcal{V}_z := \langle \mathcal{R}_z \rangle.$
\end{itemize}

\subsection[Hamiltonian vector fields as Lagrangian submanifolds]{Hamiltonian vector fields as Lagrangian and Legendrian submanifolds}

Just like in previous sections, define the \textbf{gradient vector field} of certain Hamiltonian $H\in \mathcal{C}^\infty(M)$ as $$ \grad H = \sharp_{\theta, \eta}(dH).$$ Locally, the gradient vector field is expressed: $$\grad H = \partder{H}{p_i} \partder{}{q^i} - \left (\partder{H}{q^i} + p_i\partder{H}{z} \right) \partder{}{p_i} + \left( p_i \partder{H}{p_i} + \partder{H}{z}\right ) \partder{}{z} + \partder{H}{t} \partder{}{t}.$$ We can define a symplectic structure in $TM$ taking $$\Omega_0 := \flat_{\theta, \eta} ^* \Omega_M,$$ where $\Omega_M$ is the canonical symplectic form on the cotangent bundle. In the induced coordinates $(q^i, p_i, z, \dot{q}^i, \dot{p}_i, \dot{z}),$ $\Omega_0$ takes the form 
\begin{align*}
    \Omega_0 = \,& dq^i \wedge \left (p_ip_i dq^j + (1 + \delta_i^j)p_i \dot{q}^j dp_j - \dot{z} dp_i - d\dot{p}_i - p_i d\dot{z}\right) \\
     & + dp_i \wedge d\dot{q}^i  + dz \wedge \left( d\dot{z} - p_id\dot{q}^i - \dot{q}^i dp_i\right) + dt \wedge d\dot{t}.
\end{align*}

It is easy to verify that a vector field $X: M \rightarrow TM$ is a locally gradient vector field if and only if it defines a Lagrangian submanifold in $(TM , \Omega_0).$\\

\begin{Def}[Hamiltonian vector field] Given a Hamiltonian $H$ on $M$, define its \textbf{Hamiltonian vector field} as $$X_H:= \sharp_{\theta, \eta}(dH) - (\mathcal{R}_z(H) + H) \mathcal{R}_z + (1 - \mathcal{R}_t(H)) \mathcal{R}_t.$$
\end{Def}

The Hamiltonian vector field has the local expression 
$$X_H = \partder{H}{p_i} \partder{}{q^i} - \left (\partder{H}{q^i} + p_i\partder{H}{z} \right) \partder{}{p_i} + \left( p_i \partder{H}{p_i} - H\right ) \partder{}{z} + \partder{}{t}.$$
Integral curves of this vector field satisfy
\begin{align*}
\totder{q^i}{\lambda} &= \partder{H}{p_i},\\
\totder{p_i}{\lambda} &= - \partder{H}{q^i} - p_i \partder{H}{z},\\
\totder{z}{\lambda} &= p_i \partder{H}{p_i} - H,\\
\totder{t}{\lambda} &= 1,
\end{align*}
which are equivalent to the time-dependent version of the contact equations:
\begin{align*}
\totder{q^i}{t} &= \partder{H}{p_i},\\
\totder{p_i}{t} &= - \partder{H}{q^i} - p_i \partder{H}{z},\\
\totder{z}{t} &= p_i \partder{H}{p_i} - H.
\end{align*}
In general, $X_H$ does not define a Lagrangian submanifold of $(TM, \Omega_0)$; but, just like in the cosymplectic and contact scenario, we can achieve this by modifying the symplectic form. Indeed, since $$X_H^* \Omega_0 = - dX_H^{\flat_{\theta, \eta}} =   d\left((\mathcal{R}_z(H) + H) \eta\right) - d\left ((1 - \mathcal{R}_t(H)) \theta \right),$$ defining $$ \Omega_H := \Omega_0 - \d((\mathcal{R}_z(H) + H) \eta) + \d((1 - \mathcal{R}_t(H)) \theta),$$ we have that $X_H$ defines a Lagrangian submanifold of $(TM, \Omega_H).$\\

Now we interpret the Hamiltonian vector field $X_H$ as a Legendrian submanifold of $TM \times \mathbb{R} \times \mathbb{R}$ with the cocontact structure given by the forms $$\widetilde \eta := \eta^c + s \eta^v + \theta^c + e \theta^v; \,\,\, \widetilde \theta = \theta^c,$$ where $(s,e)$ are the parameters in $\mathbb{R} \times \mathbb{R}$. In local coordinates $(q^i, p_i, z, t, \dot{q}^i, \dot{p}_i, \dot{z}, \dot{t}, s,e)$, these forms have the expression:
\begin{align*}
\widetilde \eta &= d\dot{z} - \dot{p}_i dq^{i} - p_i d\dot{q}^i + s dz - sp_i dq^i + d\dot{t} + e dt, \\
\widetilde \theta &= d\dot{t}.
\end{align*}

It is easy to see that these forms define a cocontact structure. Now, given a vector field $X: M \rightarrow TM$ and two functions $f, g$ on $M$, define $$X \times f \times g: M \rightarrow TM \times \mathbb{R} \times \mathbb{R}; \,\,\, p \mapsto (X(p), f(p), g(p)).$$
Applying the properties of complete and vertical lifts, namely $X^*\alpha^c = \mathcal{L}_X(\alpha)$, we have
$$(X \times f \times g)^* \widetilde \eta = \mathcal{L}_X \eta + f \eta + \mathcal{L}_X \theta + g \theta$$ and $$(X \times f \times g)^*\widetilde \theta = d\theta(X).$$

\begin{prop} Let $H$ be a Hamiltonian on $M$. Then $X_H \times \mathcal{R}_z(H) \times 0$ defines a Legendrian submanifold of $(TM \times \mathbb{R} \times \mathbb{R}, \widetilde \theta, \widetilde \eta).$
\end{prop}

 \begin{proof} Using the observation above and \textcolor{red}{Lemma \ref{ClassificationLegendrianCocontact}}, it is sufficient to observe that $$\mathcal{L}_{X_H}\eta = - \mathcal{R}_z(H) \eta, \,\, \mathcal{L}_{X_H}\theta =  0.$$
 \end{proof}

\subsection{Coisotropic reduction}
A cocontact manifold is also a Jacobi manifold defining $$\Lambda(\alpha, \beta) := -d\eta(\sharp(\alpha), \sharp(\beta)), \,\, E = -\mathcal{R}_z,$$ and thus, we have the $\Lambda$
-orthogonal and the corresponding definitions of isotropic, coiso\-tropic or Legendrian submanifolds and distributions.\\

Notice that $\mathcal{H}_t$ is an integrable distribution and that each leaf of its foliation inherits a contact structure. Indeed, $\mathcal{H}_t$ is the characteristic distribution defined by the Jacobi structure, $\mathcal{S}$. \\

Now we give a symplectic interpretation of the $\Lambda$-orthogonal. Notice that the restriction of $d\eta$ to $\mathcal{H}_{tz}$ defines a symplectic structure on the distribution. Denote by $\perp_{d\eta|}$ its symplectic orthogonal. The $\Lambda$-orthogonal is just the symplectic orthogonal of the intersection with $\mathcal{H}_{tz}$.\\

\begin{prop} \label{CocontactOrthogonal}Given a distribution $\Delta$ on a cocontact manifold $(M, \eta, \theta)$, $$\Delta^{\perp_\Lambda} = (\Delta \cap \mathcal{H}_{tz})^{\perp_{d\eta |}}.$$
\end{prop}
\begin{proof} We check one inclusion and compare dimensions:\\
Let $\alpha \in \Delta^0_q$ and $u \in \Delta^0_q \cap (\mathcal{H}_{tz})_q.$ We will see that $d\eta_q(u,\sharp_\Lambda(\alpha)) = 0.$ Indeed, 
\begin{align*}
 d\eta_q(u, \sharp_\Lambda(\alpha)) &= d\eta_q(u, \sharp_\Lambda(\alpha)) + \theta_q(u) \theta_q(\sharp_\Lambda(\alpha)) + \eta_q(u) \eta_q(\sharp_\Lambda(\alpha)) \\
 &= \langle \flat(u), \sharp_\Lambda(\alpha) \rangle = \Lambda_q(\alpha, \flat(u)) = - d\eta_q(\sharp(\alpha), \sharp (\flat(u)))\\
 &= - d\eta_q(\sharp(\alpha), u) = -  d\eta_q(\sharp(\alpha),u) - \theta_q(\sharp(\alpha))\theta_q(u) - \eta_q(\sharp(\alpha))\eta_q(u)\\
 & = -\langle \flat(\sharp(\alpha)), u\rangle = - \alpha(u) = 0,
\end{align*}
that is, $\Delta^{\perp_\Lambda} \subseteq (\Delta \cap (\mathcal{H}_{tz})_q)^{\perp_{ d\eta|}}.$\\

Now we compare both dimensions. Let $k:= \dim \Delta, r_q := \dim(\Delta_q^0 \cap\langle \theta_q, \eta_q\rangle).$ Since $$\Delta_q^0 \cap \langle \theta_q, \eta_q\rangle = (\Delta_q \oplus (\mathcal{H}_{tz})_q)^0,$$ we have 
\begin{align*}
r_q &= \dim (\Delta_q^0 \cap \langle \theta_q, \eta_q\rangle) = \dim(\Delta_q \oplus (\mathcal{H}_{tz})_q)^0 = 2n +2 - \dim (\Delta_q \oplus (\mathcal{H}_{tz})_q)\\
&= 2n + 2 - (\dim \Delta_q + \dim (\mathcal{H}_{tz})_q - \dim(\Delta_q \cap (\mathcal{H}_{tz})_q) ) = 2n + 2  - k - 2n + \dim (\Delta_q \cap (\mathcal{H}_{tz})_q)\\
&= 2 + \dim (\Delta_q \cap (\mathcal{H}_{tz})_q) - k,
\end{align*}
which implies that $$\dim (\Delta_q \cap (\mathcal{H}_{tz})_q) = k + r_q - 2.$$
It only remains to observe that $$\dim \Delta^{\perp_\Lambda} = 2n + 2 - k - r_q = \dim \Delta^0 - \dim (\Delta^0 \cap \ker \sharp_\Lambda)$$ and that $$\dim(\Delta \cap (\mathcal{H}_{tz})_q)^{\perp_{d\eta |}} = 2n + 2 - k - r_q = 2n - \dim(\Delta \cap (\mathcal{H}_{tz})_q). $$
\end{proof}
Now we can give a characterization of Legendrian submanifolds:
\begin{lemma} \label{ClassificationLegendrianCocontact} $i:L \rightarrow M$ is a Legendrian submanifold ($(TL)^{\perp_\Lambda} = TL$) if and only if $\dim L = n$ and $$i^*\theta = 0, i^*\eta = 0.$$
\end{lemma}
\begin{proof}Necessity is clear, since $L$ is necessarily horizontal. Sufficiency follows from $i^* d\eta = 0$ which, together with $\dim L = n$,  implies that $T_qL$ is a Lagrangian submanifold of $(\mathcal{H}_{tz})_q$, for every $q\in L$. Using \textcolor{red}{Proposition \ref{CocontactOrthogonal}}, we have the equivalence.
\end{proof}

This allows us to express the $\Lambda$-orthogonal complement of a coisotropic distribution in a more convenient way:\\
\begin{corollary} Let $\Delta$ be a coisotropic distribution on $M$ ($\Delta^{\perp_\Lambda}\subseteq \Delta$), then $$\Delta^{\perp_\Lambda} = \ker d\eta_0|_{\Delta \cap \mathcal{H}_{tz}} \cap \ker \eta_0 \cap \ker \theta_0,$$ where $\eta_0$ and $\theta_0$ are the restrictions of $\eta$ and $\theta$ to $\Delta$, respectively.\\
\end{corollary}

\begin{corollary}\label{CocontactOrthogonalCorollary} Let $i: N \hookrightarrow M$ be a coisotropic submanifold of $(M,\theta, \eta).$ Then the distribution $TN^{\perp_\Lambda}$ is involutive.
\end{corollary}
\begin{proof} Denote $\eta_0 := i^*\eta$, $\theta_0 := i^*\theta$ and let $X,Y$ be vector fields on $N$ tangent to the distribution $TN^{\perp_\Lambda}$ and $Z$ be an arbitrary $tz$-horizontal vector field on $N$. Using \textcolor{red}{Corollary \ref{CocontactOrthogonal}}, we only need to check that $[X,Y] \in \ker d\eta_0|_{\Delta \cap \mathcal{H}_{tz}} \cap \ker \eta_0 \cap \ker \theta_0.$ Indeed, expanding the expressions $0 = \d^2 \eta_0(X,Y,Z)$, $0 =  d\eta_0(X,Y)$, $0 = d\theta_0(X,Y)$; we obtain 
\begin{align*}
0 &= - d\eta_0([X,Y], Z), \\
0 &= - \eta_0([X,Y]), \\
0 &= - \theta_0([X,Y]).
\end{align*}
\end{proof}

Now, given a coisotropic submanifold $N \hookrightarrow M$, since $TN^{\perp_\Lambda}$ is involutive, it provides a maximal foliation, $\mathcal{F}$. We assume that $N/\mathcal{F}$ inherits a manifold structure such that the canonical projection $\pi: N \rightarrow N/\mathcal{F}$ is a submersion.

Just like in the previous cases, for the well-definedness and non-degeneracy of the forms in the quotient, we need to restrict the coisotropic submanifolds we are studying. Consequently, we will say that a submanifold $N \hookrightarrow M$ is:
\begin{itemize}
    \item[$i)$] \textbf{$t$-vertical} (resp. \textbf{$z$-vertical}) if $\mathcal{V}_t\subseteq TN$ (resp. $\mathcal{V}_z\subseteq TN$).
    \item[$ii))$] \textbf{$tz$-vertical}, if it is both $t$-vertical and $z$-vertical.
    \item[$iii)$] \textbf{$t$-horizontal} (resp. \textbf{$z$-horizontal}) if $\mathcal{H}_t \subseteq TN$ (resp. $\mathcal{H}_z \subseteq TN$).
    \item[$iv)$] \textbf{$tz$-horizontal} if it is both $t$-horizontal and $z$-horizontal, that is, if $TN \subseteq \mathcal{H}_{tz}$.
\end{itemize}

\subsection{\texorpdfstring{$tz-$}-vertical reduction}
Let $i: N \hookrightarrow M$ be a $tz$-vertical submanifold. It is easy to check that under these conditions $$TN^{\perp_\Lambda} = \ker  d\eta_0 \cap \ker \eta_0 \cap \ker \theta_0,$$ and that $(TN)^{\perp_\Lambda}$ is a regular distribution of rank $$\operatorname{rank}\, (TN)^{\perp_\Lambda} = \dim M - \dim N.
$$
\begin{theorem}[$tz$-vertical coisotropic reduction] Let $i: N \hookrightarrow M$ be a $tz$-vertical submanifold of a cocontact manifold $(M, \theta, \eta).$ Denote by $\mathcal{F}$ the maximal foliation induced by the integrable distribution $TN^{\perp_\Lambda}$ on $N$. If $N/\mathcal{F}$ admits a manifold structure such that the canonical projection $\pi: N \rightarrow N/\mathcal{F}$ defines a submersion, then there exists unique forms $\theta_N, \eta_N$ on $N/\mathcal{F}$ such that $(N/\mathcal{F}, \theta_N, \eta_N)$ defines a cocontact structure and $$i^*\theta = \pi^*\theta_N, \,\, i^*\eta_N = \pi^*\theta_N.$$
\end{theorem}

Before proving the theorem, let us calculate the dimension of the quotient. Let $k + 2:= \dim N$. We have $$\operatorname{rank}\, TN^{\perp_\Lambda} = 2n + 2 - (k + 2) = 2n - k$$ and, therefore, $$\dim N/\mathcal{F} = 2(k - n) + 2.$$

\begin{proof} Uniqueness is clear, since $\pi$ is a submersion. We only need to check the well-definedness taking $$\theta_N([u]) := \theta_0(u), \,\, \eta_N([u]):= \eta_0(u).$$ Independence of the vector is clear, using \textcolor{red}{Proposition \ref{CocontactOrthogonal}}. For independence on the point, let $X$ be a vector field on $N$ tangent to the distribution. It is easy to check that $$\mathcal{L}_X \eta_0 = 0 , \mathcal{L}_X \theta_0 = 0;$$ and thus, well-definedness follows. For non-degeneracy, it is enough to proof that $$\theta_0 \wedge \eta_0 \wedge (d\eta_0)^{k-n} \neq 0.$$ This follows easily from $TN^{\perp_\Lambda} = \ker d\eta_0 \cap \ker \eta_0 \cap \ker \theta_0.$
\end{proof}

\subsubsection{Projection of Legendrian submanifolds}

\begin{prop} \label{ProjectionLegendrianTZvertical}Let $L \hookrightarrow M$ be a Legendrian submanifold and $i: L \hookrightarrow M$ be a $tz$-vertical coisotropic submanifold. If $L$ and $N$ have clean intersection and $L_N 
= \pi(L \cap N)$ is a submanifold in $N/\mathcal{F}$, $L_N$ is Legendrian.
\end{prop}
\begin{proof} Using \textcolor{red}{Lemma \ref{ClassificationLegendrianCocontact}}, $L_N$ is clearly isotropic. Now we need to check that $\dim L_N = k - n$, given that $\dim N/\mathcal{F} = 2(k-n) + 2.$  We have 
\begin{equation}\label{eq0cocontact}
    \dim \pi(L \cap N) = \dim(L \cap N) - \operatorname{rank}\, (TL \cap (TN)^{\perp_\Lambda}).
\end{equation}
Furthermore, since $(TL \cap (TN)^{\perp_\Lambda}) = TL + (TN^{\perp_\Lambda})^{\perp_\Lambda}$ and $TL \cap (TN)^{\perp_\Lambda}$ is $tz$-horizontal, 

\begin{equation}\label{eq1cocontact}
    \operatorname{rank} \,(TL + (TN^{\perp_\Lambda})^{\perp_\Lambda}) = 2n - \operatorname{rank}\, (TL \cap TN^{\perp_\Lambda}).
\end{equation}
Now, using the Grassman formula:
\begin{align}
    \operatorname{rank} \,(TL + (TN^{\perp_\Lambda})^{\perp_\Lambda}) &= \dim L + \operatorname{rank} \, (TN^{\perp_\Lambda})^{\perp_\Lambda} - \operatorname{rank} \, (TL \cap (TN^{\perp_\Lambda})^{\perp_\Lambda})\\
    &= \dim L + (\dim N - 2) - \dim(L \cap N)\\& = n + k - \dim(L \cap N). \label{eq2cocontact}
\end{align}
From (\ref{eq1cocontact}) and (\ref{eq2cocontact}) we obtain 
\begin{equation}
    \operatorname{rank} \,(TL \cap TN^{\perp_\Lambda}) = n - k + \dim(L \cap N).
\end{equation}
Substituting in (\ref{eq0cocontact}) yields $\dim \pi(L \cap N) = k - n.$
\end{proof}
\subsection{\texorpdfstring{$t-$}-vertical, \texorpdfstring{$z-$}-horizontal reduction}
Suppose $i: N \hookrightarrow M$ is a $t$-vertical and $z$-horizontal coisotropic submanifold. This time we have $$(TN)^{\perp_\Lambda} = \ker \theta_0,$$ since $\eta_0 = 0$ implies $d\eta _0 = 0$. We conclude that $(TN)^{\perp_\Lambda} = TN \cap \mathcal{H}_{tz}$, which implies that $$\dim N/\mathcal{F} = 1.$$ This means that reduction is trivial, leaving the trivial cosymplectic submanifold of dimension $1$:\\
\begin{theorem} Let $i: N \hookrightarrow M$ be a $t$-vertical, $z$-horizontal coisotropic submanifold of a coconatct manifold $(M , \theta, \eta)$. Denote by $\mathcal{F}$ the maximal foliation defined by the distribution $(TN)^{\perp_\Lambda}$. If $N/\mathcal{F}$ has a manifold structure such that the canonical projection $\pi: N \rightarrow N/\mathcal{F}$ defines a submersion, then $N/\mathcal{F}$ is one-dimensional and there exists and unique volume form $\theta_N$ on $N/\mathcal{F}$ such that $$i^* \theta = \pi^* \theta_N.$$
\end{theorem}

\begin{remark}{\rm Given the triviality of reduction in the $t$-vertical and $z$-horizontal case, projection of Legendrian submanifolds in $M$ will always result in $0$-dimensional Lagrangian submanifolds in $N/\mathcal{F}$. }
\end{remark}

\subsection{\texorpdfstring{$z-$}-vertical, \texorpdfstring{$t-$}-horizontal reduction} 
Let $i: N \rightarrow M$ be a $z$-vertical and $t$-horizontal coisotropic submanifold. It is easy to check that this time we have the equality 
$$TN = \ker d\eta_0 \cap \ker \eta_0.$$
Since $\mathcal{H}_t$ is integrable, we have that coisotropic reduction of $N$ is actually happening in one of the leaves of the foliation that inhertits a contact structure from the cocontact structure. We conclude, from \textcolor{red}{Theorem \ref{VerticalReductionContact}}:\\

\begin{theorem} Let $i: N \rightarrow M$ be a $z$-vertical and $t$-horizontal coisotropic submanifold of a cocontact manifold $(M, \theta, \eta)$. Denote by $\mathcal{F}$ the maximal foliation on $N$ defined by the distribution $TN^{\perp_\Lambda}$. If $N/\mathcal{F}$ has a manifold structure such that the canonical projection $\pi: N \rightarrow N/\mathcal{F}$ defines a submersion, then there exists an unique form $\eta_N$ such that $(N/\mathcal{F}, \eta_N)$ is a contact manifold and $$i^*\eta = \pi^* \eta_N.$$
\end{theorem}
\subsubsection{Projection of Legendrian submanifolds}
\begin{prop}Let $L \hookrightarrow M$ be a Legendrian submanifold. If $L$ and $N$ have clean intersection and $L_N = \pi(L \cap N)$ is a submanifold in $N/\mathcal{F}$, $L_N$ is Legendrian in $(N/\mathcal{F}, \theta_N)$.
\end{prop}
\begin{proof} It is clearly horizontal and, therefore, using \textcolor{red}{Lemma \ref{isotropicClassificationContact}}, it is isotropic. Now, supposing $k = \dim N$, we only need to check that $$\dim L_N = k - n,$$ since $\dim N/\mathcal{F} = 2(k - n) + 1.$ This is straight-forward, following the same steps given in  \textcolor{red}{Proposition \ref{ProjectionLegendrianTZvertical}}.
\end{proof}

\subsection{\texorpdfstring{$tz-$}-horizontal reduction}
Let $i: N \hookrightarrow M$ be a $tz$-horizontal coisotropic submanifold. Since $\eta_0 = 0$, $d\eta_0 = 0$ and $\theta_0 = 0$, we have $$(TN)^{\perp_\Lambda} = TN,$$ wich implies that $$\dim N/\mathcal{F} = 0,$$ leaving a trivial symplectic manifold, having as many points as path components of $N$. This means that if $N/\mathcal{F}$ admits a manifold structure, it will be a symplectic manifold. \\

\begin{remark}{\rm Proceeding as in the previous cases we immediately obtain that the projection of Legendrian submanifolds is trivial.}
\end{remark}





\section{ Coisotropic reduction in stable Hamiltonian structures}\label{SHS}
\setsectiontitle{COISOTROPIC REDUCTION IN STABLE HAMILTONIAN STRUCTURES}
There were several attempts to combine cosymplectic and contact structures. The first one is due to Albert \cites{albert1989theoreme}, using a combination of a 1-form and a 2-form; however, the setting is not useful for us since the lack of integrability. The second attempt in this direction is due to Acakpo \cites{acakpo2022stable}, which is studied in this section.\\

\begin{Def}[Stable Hamiltonian structure] A \textbf{stable Hamiltonian structure} (SHS) is a triple $(M, \omega, \lambda)$ where $M$ is a $2n + 1$ dimensional manifold, $\omega$ is a closed $2$-form and $\lambda$ is a $1$-form such that $$\lambda \wedge \omega^n \neq 0, \,\,\, \ker \omega \subseteq \ker d\lambda.$$
\end{Def}

There exists, just like in the previous cases, a natural isomorphism $$\flat_{\lambda, \omega}: TM \rightarrow T^*M;\,\,\,v_q \mapsto i_{vq} \omega + \lambda(v_q) \cdot \lambda.$$ and its inverse $\sharp_{\lambda, \omega} := \flat_{\lambda, \omega} ^{-1}.$ Let us perform some calculations in coordinates. Since $d\omega$ is closed and of constant range $2n$, around any point, there exists coordinates $(q^i, p_i, z)$ such that $$\omega = dq^i \wedge dp_i$$ (see \cites{godbillon1969geometrie}). In this coordinate chart $\lambda$ will have an expression of the form $$\lambda = a_i dq^i + b^i dp_i + c dz.$$ Since $$0 \neq \lambda \wedge \omega^n = c dz \wedge \omega^n,$$ we conclude that $c \neq 0$. Let $\varphi_t(q^i, p_i, z)$ be the (local) flow of the vector field $\displaystyle{\frac{1}{c} \partder{}{z}}$. Fix some value $z_0$, and define the map $$\psi(q^i, p_i, t) := (q^i, p_i, z(\varphi_t(q^i, p_i, z_0))).$$ It is clear that this defines a local diffeomorphism. Take the new set of coordinates to be $$(q^i, p_i, t) := \psi^{-1} \circ (q^i, p_i, z).$$ We have $$dz = \partder{z}{q^i} dq^i + \partder{z}{p_i} dp_i + \partder{z}{t} dt = \partder{z}{q^i} dq^i + \partder{z}{p_i} dp_i + \frac{1}{c} dt.$$ Therefore, in the new coordinate chart, $$\lambda = \left( a_i + \partder{z}{q^i}\right) dq^i + \left( b_i + \partder{z}{p_i}\right) dp_i + dt.$$ We conclude:

\begin{prop} Around every point of $M$ there exists a coordinate chart $(q^i, p_i, z)$ such that $$\omega = dq^i \wedge dp_i, \lambda = a_i dq^i + b^i dp_i + dz.$$ We call these coordinates Darboux coordinates.
\end{prop}

In Darboux coordinates, the condition $\ker  d\omega \subseteq d\lambda$ translates to $$\partder{a_ i}{z} = \partder{b^i}{z} = 0.$$ Also, the musical isomorphisms take the expression:
\begin{align*}
    \flat_{\lambda, \omega}\left (\partder{}{q^i}\right ) &= dp_i + a_i a_j dq^j + a_i b^j dp_j + a_i dz,\\
    \flat_{\lambda, \omega} \left (\partder{}{p_i}\right ) &= -dq^i + b^i a_j dq^j + b^ib^j dp_j + b^i dz,\\
    \flat_{\lambda, \omega}\left (\partder{}{z} \right) &= a_ i dq^i + b^i dp_i + dz,\\
    \sharp_{\lambda, \omega} \left( dq^i\right) &= - \partder{}{p_i} + {b^i} \partder{}{z},\\
    \sharp_{\lambda, \omega} \left( dp_i\right) &= \partder{}{q^i} - {a_ i} \partder{}{z},\\
    \sharp_{\lambda, \omega} \left(dz \right) &= \partder{}{z} + {a_i} \partder{}{p_i} - {b^i} \partder{}{q^i}.\\
\end{align*}
Imitating the definitions in the contact case, we can define a bivector field on $M$ as $$\Lambda_q(\alpha_q, \beta_q) := \omega_q(\sharp_{\lambda, \omega}(\alpha_q), \sharp_{\lambda, \omega}(\beta_q)),$$ and the morphism $$\sharp_\Lambda: T^*M \rightarrow TM; \,\,\, \alpha_q \mapsto i_{\alpha_q} \Lambda$$ with the induced $\Lambda$-orthogonal complement for distributions $$\Delta_q^{\perp_\Lambda} := \sharp_\Lambda(\Delta_q^0).$$ In coordinates $(q^i, p_i, z)$ the bivector field $\Lambda$ takes the local form: $$\Lambda = \partder{}{q^i} \wedge \partder{}{p_ i} + \left( {a_ i} \partder{}{p_ i} - {b^i} \partder{}{q^i} \right) \wedge  \partder{}{z}.$$
We also have the distributions 
\begin{itemize}
\item[i)] $\mathcal{H}_q:= \ker \lambda_q,$
\item[ii)] $ \mathcal{V}_q:= \ker \omega_q,$
\end{itemize}
and the Reeb vector field $\mathcal{R}_q := \sharp_{\lambda, \omega}(\lambda_q).$ Locally, 
\begin{align}
    \mathcal{H} &= \langle \partder{}{q^i} - {a_i} \partder{}{z}, \partder{}{p_i} - {b^i} \partder{}{z}\rangle,\\
    \mathcal{V} &= \langle  \partder{}{z}\rangle, \\
    \mathcal{R} &=  \partder{}{z}.
\end{align}

A natural question to ask is wether the bivector field $\Lambda$ arises from a Jacobi bracket. We have 
$$[\Lambda, \Lambda] = 2 \partder{}{p_i} \wedge \left ( \partder{a_j}{q^i} \partder{}{p_j} - \partder{b^j}{q^i} \partder{}{q^j}  \right) \wedge \partder{}{z} - 2 \partder{}{q^i} \wedge \left( \partder{a_j}{p_i} \partder{}{p_j} - \partder{b^j}{p_i}\partder{}{q^j} \right) \wedge \partder{}{z}.$$ Taking an arbitrary vector field $$E = X^i \partder{}{q^i} + Y_i \partder{}{p_i} + Z \partder{}{z},$$ $(\Lambda, E)$ defines a Jacobi structure if and only if $$[\Lambda, \Lambda] = 2 E \wedge \Lambda, \,\,\,[E, \Lambda] = 0.$$ It is easily checked that the first equality holds when
\begin{align}
&\partder{a_j}{q^i} - \partder{a_ i}{q^j} = \partder{b^j}{p_i} - \partder{b^i}{p_j} = 0,\\
&  \partder{b^i}{q^j} -\partder{a_j}{p_i} = 0, \,\, i \neq j,\\
&  \partder{b^i}{q_i} - \partder{a_i}{p_i}= f,\\
&X^i = Y_i = 0,\\
&Z = f,
\end{align}
for certain local unique funtion $f$. It is easy to check that this relations translate intrinsically to $$d\lambda = f \omega, \,\, E =  f \mathcal{R}.$$
Now, let us compute $[E, \Lambda]$ for $E = f \displaystyle{\partder{}{z}}.$
$$[E, \Lambda] = \left ( \partder{f}{q^i} - a_i \partder{f}{z}\right) \partder{}{p_i} \wedge \partder{}{z} + \left (- \partder{f}{p_i} + b^i \partder{f}{z} \right ) \partder{}{q^i} \wedge \partder{}{z}.$$
Therefore, $[E, \Lambda] = 0$ if and only if 
\begin{align}
    \partder{f}{q^i} - a_i \partder{f}{z} = -\partder{f}{p_i} + b^i \partder{f}{z} = 0.
\end{align}
This is easily seen to be equivalent to $$\sharp_{\lambda, \omega}(df) \in \mathcal{V}.$$
We have concluded the following:\\
\begin{prop} The bivector field $\Lambda$ arises from a Jacobi structure if and only if there exists some $f \in \mathcal{C}^\infty(M)$ such that $$d\lambda = f \omega, \,\,\ \sharp_{\lambda, \omega}(df) \in \mathcal{V}.$$ And, in that case, the Jacobi structure is defined by the pair $(\Lambda, f \mathcal{R})$.\\
\end{prop}
\begin{remark} {\rm Notice that we recover the cosymplectic scenario when $f = 0$ and the contact scenario when $f = 1$ (because the definition of $\Lambda$ in contact geometry is the opposite of the definition we gave in SHS).}
\end{remark}

Let us return to the study of coisotropic reduction. It is easy to see that $\omega$ induces a symplectic form in $\mathcal{H}$, $\omega|_\mathcal{H}$. This induces the symplectic orthogonal for $\Delta_q \subset \mathcal{H}_q$: $$\Delta_q^{\perp_{\omega|_\mathcal{H}}} .= \{v \in \mathcal{H}_q\,| \, \omega(v,w) = 0 \,\forall w \in \Delta_q\}.$$

A distribution $\Delta$ in $M$ will be called:
\begin{itemize}
\item[i)] \textbf{Isotropic} if $\Delta \subseteq \Delta^{\perp_\Lambda};$
\item[ii)] \textbf{Coisotropic} if $\Delta^{\perp_\Lambda} \subseteq \Delta;$
\item[iii)] \textbf{Lagrangian} if $\Delta^{\perp_\Lambda} = \Delta^{\perp_\Lambda} \cap \mathcal{H}.$
\end{itemize}

We have the following equality:\\

\begin{prop} Let $\Delta$ be a distribution on $M$. Then $$\Delta^{\perp_\Lambda} = (\Delta \cap \mathcal{H})^{\perp_{\omega|_\mathcal{H}}}.$$
\end{prop}
\begin{proof} The proof follows the same lines as that of \textcolor{red}{Proposition \ref{CosymplecticOrthognal}}
\end{proof}

Just like in previous sections, we say that a Lagrangian submanifold $L \hookrightarrow M$ is \textbf{horizontal} if $T_qL \subseteq \mathcal{H}_q\, \forall q \in L$ and say that it is \textbf{non-horizontal} if $T_qL \not \subseteq  \mathcal{H}_q \, \forall q \in L.$ We have the following characterization:\\

\begin{lemma} Let $L \hookrightarrow M$ be an isotropic (or coisotropic) submanifold. We have
\begin{itemize}
\item[i)] If $L$ is horizontal and $\dim L = n$, then $L$ is Lagrangian.
\item[ii)] If $L$ is non-horizontal and $\dim L = n+1$, then $L$ is Lagrangian.
\end{itemize}
\end{lemma}
\begin{proof} The proof is similar to the proof of \textcolor{red}{Lemma \ref{CharacterizationLagrangianCosymplectic}}, since we only need to check the condition at each tangent space.
\end{proof}

Now, given a coisotropic submanifold $N \hookrightarrow M$ (that is, $(T_qN)^{\perp_\Lambda} \subseteq T_qN$), the distribution $(TN)^{\perp_\Lambda}$ is not necessarily integrable and we shall assume it in what follows:
\subsection[Hamiltonian vector fields as Lagrangian submanifolds]{Gradient and Hamiltonian vector fields as Lagrangian submanifolds}
We can define a symplectic structure on $TM$ taking $$\Omega_0 := \flat_{\lambda, \omega} ^* \Omega_M,$$ where $\Omega_M$ is the canonical symplectic form on $T^*M$. In Darboux coordinates it has the expression:
\begin{align*}
  \Omega_0 =& dq^j \wedge d\left (\dot{q}^ia_ia_j + \dot{p}_ib^ia_j+ \dot{z}a_j - \dot{p}_j \right)  + \\
   & dp_j \wedge d\left( \dot{q}^ia_ib^j + \dot{p_i}b^ib^j + \dot{z}b^j + \dot{q}^j \right)+ \\
   &  dz \wedge d\left( \dot{q}^ia_i + \dot{p}_ib^i + \dot{z}\right).
\end{align*}

\begin{Def}[Gradient vector field] Given a Hamiltonian $H \in \mathcal{C}^ \infty(M),$ define the \textbf{gradient vector field} of $H$ as $$\grad H:= \sharp_{\lambda, \omega}(dH).$$
\end{Def}
In Darboux coordinates, the gradient vector field is written

$$\grad H = \left ( \partder{H}{p_i} - b^i \partder{H}{z} \right) \partder{}{q^i} + \left( - \partder{H}{q^i} + a_i \partder{H}{z}\right) \partder{}{p_i} + \left( b^i\partder{H}{q^i} - a_i\partder{H}{p_i} + \partder{H}{z}\right) \partder{}{z}.$$

It is easily checked that $X: M \rightarrow TM$ is locally a gradient vector field if and only if $X(M)$ is a Lagrangian submanifold of $(TM, \Omega_0)$. Indeed, we have the equality $$X^*\Omega_0 = - d\flat_{\lambda, \omega}(X).$$

When $\Lambda$ comes from a Jacobi bracket on $M$, that is, when $$d\lambda = f \omega, \,\,\, [f \mathcal{R}, \Lambda] = 0,$$ for some function $f$ on $M$, we have the Hamiltonian vector field of the Jacobi structure $(\Lambda, f \mathcal{R})$: 
$$X_H = \sharp_\Lambda(dH ) + f H \mathcal{R} = -\grad H + (\mathcal{R}(H) + fH )\mathcal{R} .$$ In Darboux coordinates it has the expression:
$$X_H = \left ( - \partder{H}{p_i}+ b^i \partder{H}{q^i}\right) \partder{}{q^i} + \left( \partder{H}{q^i} - a_i \partder{H}{z}\right) \partder{}{p_i} + \left( a_i \partder{H}{p_i} - b^i \partder{H}{q^i} + f H\right) \partder{}
{z}.$$

Let us interpret the Hamiltonian vector field as a Lagrangian submanifold of $TM$, with some symplectic form. First, observe that
\begin{align*}
    X_H ^* \Omega_0 = - d(\flat_{\lambda, \omega}(dH)) = - d\left(  \mathcal{R}(H) \lambda + fH \lambda\right).
\end{align*}
Therefore, defining the symplectic form $$\Omega_H := \Omega_0 + d(\mathcal{R}(H)  \lambda + fH \lambda)^v,$$ we have that $X_H$ defines a Lagrangian submanifold of $(TM, \Omega_H).$

 \subsection{Vertical coisotropic reduction}


\begin{theorem}[Vertical coisotropic reduction in stable Hamiltonian structures] Let $i: N \hookrightarrow M$ be a vertical coisotropic submanifold such that $(TN)^{\perp_\Lambda}$ defines an integrable distribution. Let $\mathcal{F}$ be the set of leaves and suppose that $N/\mathcal{F}$ admits a manifold structure such that the canonical projection $\pi: N \rightarrow N/\mathcal{F}$ defines a submersion. If $i^* d\lambda = 0$ in $TN \cap \mathcal{H}$, then $N/\mathcal{F}$ admits an unique stable Hamiltonian system structure $(\omega_N, \lambda_N)$ such that $\pi^*\omega_N = i^*\omega$ and $\pi^*\lambda_N = i^*\lambda$. The following diagram summarizes the situation:
 \[
\begin{tikzcd} N \arrow[r, "i"] \arrow[d, "\pi"] & M \\ N / \mathcal{F} \end{tikzcd} \]
\end{theorem}
\begin{proof} The proof is similar to the the proof of \textcolor{red}{ Theorem \ref{VerticalCosymplecticReduction}}. Asking $i^*d\lambda = 0$ is necessary to guarantee the well-definedness of $\lambda_N$ in the quotient using $$\mathcal{L}_X \lambda_0 = i_Xd\lambda-0 + di_X\lambda_0 = 0,$$ where $\lambda_0 = i^*\lambda$. It would only remain to check that $$\ker \omega_N \subseteq \ker d\lambda_N.$$ Indeed, since $\ker \omega_N = \langle \mathcal{\mathcal{R}_N} \rangle$, and $\mathcal{R}_N = \pi_* \mathcal{R}$, it follows from $$\pi^*(i_{\mathcal{R}_N}d\lambda_N) = i_\mathcal{R} d\lambda =  0.$$
\end{proof}

\subsubsection{Projection of Lagrangian submanifolds}
We have the result:\\
\begin{prop}[Projection of Lagrangian submanifolds] Let $i: L \hookrightarrow M$ be a Lagrangian submanifold. If $L$ and $N$ have clean intersection and $\pi(L \cap N)$ is a submanifold in $N /\mathcal{F}$, then it is Lagrangian.
\end{prop}
\begin{proof} The proof is similar to the the proof of \textcolor{red}{Proposition \ref{ProjectionHorizontalLagrangianCosymplectic}} and \textcolor{red}{Proposition \ref{ProjectionNonHorizontalLagrangian}}, since the proof reduces to the study of each tangent space.
\end{proof}

\subsection{Horizontal coisotropic reduction}

\begin{theorem}[Horizontal coisotropic reduction in stable Hamiltonian structures] Let $i: N \rightarrow M $ be a coisotropic horizontal submanifold such that $(TN)^{\perp_\Lambda}$ defines an integrable distribution. Let $\mathcal{F}$ be the set of leaves of the foliation and suppose that $N/\mathcal{F}$ admits a manifold structure such that the canonical projection $\pi: N \rightarrow N/\mathcal{F}$ defines a submersion. Then $N/\mathcal{F}$ admits and unique symplectic structure $\omega_N$ such that $\pi^*\omega_N = i^*\omega$. The following diagram summarizes the situation:
 \[
\begin{tikzcd} N \arrow[r, "i"] \arrow[d, "\pi"] & M \\ N / \mathcal{F} \end{tikzcd} \]
\end{theorem}
\begin{proof}  The proof is similar to the the proof of \textcolor{red}{Theorem \ref{SymplecticReduction}}.
\end{proof}

\subsubsection{Projection of Lagrangian submanifolds}
\begin{prop}[Projection of Lagrangian submanifolds] Let $L \hookrightarrow M$ be a Lagrangian submanifold. If $L$ and $N$ have clean intersection and $\pi(L \cap N)$ is a submanifold in $N/\mathcal{L}$, then it is Lagrangian.
\end{prop}
\begin{proof} The proof is similar to the the proof of \textcolor{red}{Proposition \ref{LagrangianProjectionSymplectic}} since we only need to check it in every tangent space.
\end{proof}

\section{Conclusions}
\setsectiontitle{CONCLUSIONS}
In this paper we have reviewed the concept of coisotropic and Lagrangian (and Legendrian) submanifolds in different geometric settings. We have shown the connection of these geometric constructions with the different dynamics that usually appear in mechanics. All these different geometries (symplectic, cosymplectic, contact, cocontact) can be classified within Jacobi geometry (in some case, Poisson, a particular case of Jacobi structure). This approach allows in a simple way to see these particular situations from a more general point of view. Sometimes it is important to stand at a certain altitude in order to realise that these particular situations respond to the same geometrical pattern.

In each case we have introduced the notions of coisotropic and Lagrangian submanifold and studied in detail the corresponding coisotropic reduction theorems. The interpretations of the different types of vector fields in the different types of geometry as Lagrangian or Legendrian submanifolds are summarized in \textcolor{red}{Table \ref{Table1}}. Also, the results on coisotropic reduction are summarized in \textcolor{red}{Table \ref{Table2}}

\begin{table}[ht]
\centering
\caption{\small \textbf{Interpretation of vector fields as Lagrangian or Legendrian submanifolds}}
\rule{0pt}{4ex}  
\resizebox{\textwidth}{!}{%
  \begin{tabular}{llll}
  \label{Table1}
    Geometry & Vector field & Type of submanifold & Ambient manifold\\
    [+0.3ex]\\
    \hline\\
    \textbf{Symplectic} & Hamiltonian & Lagrangian & $(TM, \omega_0)$, symplectic\\
      $(M, \omega)$& $X_H = \sharp_\omega(dH)$ && $\omega_0 = \flat_\omega^* \omega_M$ \\
     [+0.3ex]\\
     \hline\\

     \textbf{Cosymplectic} & Gradient & Lagrangian & $(TM, \Omega_0)$, symplectic\\
     $(M, \Omega, \theta)$&  $\grad H = \sharp_{\theta, \Omega}(dH)$ && $\Omega_0 = \flat_{\theta, \Omega}^* \Omega_M$\\
     [+0.2ex]\\
     & Hamiltonian & Lagrangian & $(TM, \Omega_H)$, symplectic\\ & $X_H = \grad H - \mathcal{R}(H) \mathcal{R}$ && $\Omega_H = \Omega_0 + (\d(\mathcal{R}(H)) \wedge \theta)^v$\\
     [+0.2ex]\\
     && Lagrangian & $(TM \times \mathbb{R},\Omega_H,ds)$, cosymplectic\\
     [+0.2ex]\\
     & Evolution & Lagrangian & $(TM, \Omega_H)$, symplectic\\ & $\mathcal{E}_H = \grad H + \mathcal{R}$\\
     [+0.3 ex]\\
     \hline\\
     \textbf{Contact} & Gradient & Lagrangian & $(TM, \Omega_0)$, symplectic\\ $(M, \eta)$& $\grad H = \sharp_{\eta}(dH)$ && $\Omega_0 = \flat_\eta ^*\Omega_M$ \\
     [+0.2ex]\\
     & Hamiltonian & Lagrangian & $(TM, \Omega_H)$, symplectic \\ &$X_H = \grad H - (\mathcal{R}(H) + H)\mathcal{R}$ && $\Omega_H = \Omega_0 - \d(\mathcal{R}(H) \eta + H \eta) ^v$ \\ 
     [+0.2ex]\\
     && Legendrian & $(TM \times \mathbb{R}, \hat{\eta})$, contact\\
     &&& $\hat{\eta} = \eta^c + t \eta^v$\\
     [+0.2ex]\\
     & Evolution & Lagrangian & $(TM\times \mathbb{R}, \widetilde \Omega_H)$, symplectic\\ & $\mathcal{E}_H = X_H + H \mathcal{R}$  && $\widetilde \Omega_H = \Omega_0 - \d(\mathcal{R}(H) \eta)^v$\\
     [+0.3ex]\\
     \hline\\
     \textbf{Cocontact} & Gradient & Lagrangian & $(TM, \Omega_0)$, symplectic\\
     $(M, \theta, \eta)$& $\grad H = \sharp_{\theta, \eta}(dH)$ && $\Omega_0 = \flat_{\eta, \theta}^* \Omega_M$\\
     [+0.2ex]\\
     &Hamiltonian & Lagrangian & $(TM, \Omega_H)$, symplectic\\
     & $X_H = \grad(H) - (\mathcal{R}_z(H) + H)\mathcal{R}_z$ && $\Omega_H = \Omega_0 - \d(\mathcal{R}_z(H) \eta + H\eta)^v$\\
     & $\quad \quad \quad+(1 - \mathcal{R}_t(H))\mathcal{R}_t$ && $\quad \quad \quad + d(\theta - \mathcal{R}_t(H)\theta)^v$\\
     [+0.2ex]\\
     && Legendrian & $(TM \times \mathbb{R} \times \mathbb{R}, \widetilde \theta, \widetilde \eta)$, cocontact\\
     &&& $\widetilde \eta = \eta^c + s \eta ^v + \theta^c + e \theta ^v$\\
     &&& $\widetilde \theta = \theta^c$\\
     [+0.3ex]\\
     \hline\\
     \textbf{SHS} & Gradient & Lagrangian & $(TM, \Omega_0)$, symplectic\\
     $(M , \omega, \lambda)$ & $\grad H = \sharp_{\lambda, \omega}({dH})$ & & $\Omega_0 = \flat_{\lambda, \omega} ^*\Omega_M$\\
     [+0.2ex]\\
     If $d\lambda = f \omega,$ & Hamiltonian & Lagrangian & $(TM, \Omega_H)$, symplectic\\
      $\sharp_{\lambda, \Omega}(df) \in \mathcal{V} $& $X_H = -\grad H + (\mathcal{R}(H) + f H) \mathcal{R}$ && $\Omega_H = \Omega_0 + \d(\mathcal{R}(H)\lambda + fH\lambda)^v$\\
  \end{tabular}}
\end{table}

\begin{table}[ht]
\centering
\caption{\small \textbf{Summary of results on coisotropic reduction}}
\rule{0pt}{4ex}  
\resizebox{\textwidth}{!}{%
  \begin{tabular}{llll}
  \label{Table2}
    Geometry & Coisotropic submanifold & Resulting manifold & Projection of Lagrangian\\ & $N \hookrightarrow M$&& and Legendrian submanifolds\\
    [+0.3 ex]\\
    \hline\\
    \textbf{Symplectic} & Arbitrary & $(N/\mathcal{F}, \omega_N)$, symplectic  & $L$ Lagrangian $\implies$ $L_N$ Lagrangian\\
    [+0.2ex]\\
    \hline\\
    \textbf{Cosymplectic} & Vertical & $(N/\mathcal{F}, \theta_N, \Omega_N)$, cosymplectic & $L$ Lagrangian $\implies$ $L_N$ Lagrangian\\
    [+0.2ex]\\
    & Horizontal & $(N, \Omega_N)$, symplectic & $L$ Lagrangian $\implies$ $L_N$ Lagrangian\\
    [+0.2ex]\\
    & Arbitrary & Foliation consisting of symplectic\\
    && manifolds of $N/\mathcal{F}$\\
    [+0.3ex]\\
    \hline\\
    \textbf{Contact} & Vertical  & $(N/\mathcal{F}, \eta_N)$, contact &$L$ Legendrian $\implies$ $L_N$ Legendrian\\
    [+0.2ex]\\
    & Horizontal & $\dim N/\mathcal{F} = 0$ & \\
    [+0.3ex]\\
    \hline\\
    \textbf{Cocontact} & $tz$-vertical & $(N/\mathcal{F}, \theta_N, \eta_N)$, cocontact & $L$ Legendrian $\implies$ $L_N$ Legendrian\\
    [+0.2ex]\\
    & $t$-vertical, $z$-horizontal & $\dim N/\mathcal{F} = 1$, $\theta_N \neq 0$\\
    [+0.2ex]\\
    & $z$-vertical, $t$-horizontal & $(N/\mathcal{F}, \eta_N),$ contact& $L$ Legendrian $\implies$ $L_N$ Legendrian\\
    [+0.2ex]\\
    & $tz$-horizontal & $\dim N/\mathcal{F} = 0$ \\
    [+0.3ex]\\
    \hline\\
    \textbf{SHS} & Vertical & $(N/\mathcal{F}, \omega_N, \lambda_N)$, stable Hamiltonian & $L$ Lagrangian $\implies$ $L_N$ Lagrangian \\
    \textbf{} & \\
    \textbf{} & Horizontal & $(N/\mathcal{F}, \omega_N)$, symplectic & $L$ Lagrangian $\implies$ $L_N$ Lagrangian
  \end{tabular}}
\end{table}

\section{Acknowledgements}
We acknowledge the financial support of Grant PID2019-106715GBC21, the Severo Ochoa Programme for Centres of Excellence in R\&D (CEX2019-000904-S), and JAE Intro Programme 2022 (Becas de Introducción a la Investigación para estudiantes universitarios). Finally, we also thank the referees for their corrections and suggestions.

\phantomsection
\addcontentsline{toc}{section}{References}
\pagestyle{empty}
\bibliographystyle{plain} 
\bibliography{refs} 

\begin{bibdiv}
\begin{biblist}

\bib{abraham2008foundations}{book}{
      author={Abraham, R.},
      author={Marsden, J.E.},
       title={Foundations of {M}echanics},
   publisher={The Benjamin/Cummings Publishing Company, Reading,
  Massachusetts},
        date={1978},
      number={364},
}

\bib{acakpo2022stable}{article}{
      author={Acakpo, B.},
       title={Stable {H}amiltonian structure and basic cohomology},
        date={2022},
     journal={Annali di Matematica Pura ed Applicata (1923-)},
      volume={201},
      number={5},
       pages={2465\ndash 2470},
}

\bib{albert1989theoreme}{article}{
      author={Albert, C.},
       title={Le th{\'e}oreme de r{\'e}duction de {M}arsden-{W}einstein en
  g{\'e}om{\'e}trie cosymplectique et de contact},
        date={1989},
     journal={Journal of Geometry and Physics},
      volume={6},
      number={4},
       pages={627\ndash 649},
}

\bib{ArnoldThermo}{inproceedings}{
      author={Arnold, V.I.},
       title={Contact geometry: The geometrical method of {G}ibbs'
  thermodynamics},
        date={1989},
   booktitle={Gibbs symposium videorecordings (ru 115)},
}

\bib{arnold}{book}{
      author={Arnold, V.I},
       title={Mathematical {M}ethods of {C}lassical {M}echanics},
   publisher={Grad. Texts in Math., 60, Springer-Verlag, New York, xvi+516
  pp.},
        date={1995},
}

\bib{binz}{book}{
      author={Binz, E.},
      author={Śniatycki, J.},
      author={Fischer, H.},
       title={Geometry of classical fields},
   publisher={North-Holland Publishing Co.},
        date={1988},
      volume={154},
}

\bib{BOJOWALD_2003}{article}{
      author={Bojowald, M.},
      author={Strobl, T.},
       title={Poisson geometry in constrained systems},
        date={2003-09},
        ISSN={1793-6659},
     journal={Reviews in Mathematical Physics},
      volume={15},
      number={07},
       pages={663–703},
         url={http://dx.doi.org/10.1142/S0129055X0300176X},
}

\bib{bravetti0}{article}{
      author={Bravetti, A.},
       title={Contact {H}amiltonian dynamics: the concept and its use},
        date={2017},
     journal={Entropy},
      volume={19},
      number={10, 535},
       pages={12 pp.},
}

\bib{bravetti2}{article}{
      author={Bravetti, A.},
       title={Contact geometry and thermodynamics},
        date={2019},
     journal={Int. J. Geom. Methods Mod. Phys.},
      volume={16},
       pages={51 pp.},
}

\bib{bravetti1}{article}{
      author={Bravetti, A.},
      author={Lopez-Monsalvo, C.~S.},
      author={Nettel, F.},
       title={Contact symmetries and {H}amiltonian thermodynamics},
        date={2015},
     journal={Ann. Physics},
      volume={361},
       pages={377–400},
}

\bib{cantrijn1992gradient}{article}{
      author={Cantrijn, F.},
      author={de~Le{\'o}n, M.},
      author={Lacomba, E.A.},
       title={Gradient vector fields on cosymplectic manifolds},
        date={1992},
     journal={Journal of Physics A: Mathematical and General},
      volume={25},
      number={1},
       pages={175},
}

\bib{catalanes}{article}{
      author={Cariñena, J.~F.},
      author={Gràcia, X.},
      author={Marmo, G.},
      author={Martínez, E.},
      author={Muñoz-Lecanda, M.~C.},
      author={Román-Roy, N.},
       title={Geometric {H}amilton-{J}acobi theory},
        date={2006},
     journal={Int. J. Geom. Methods Mod. Phys.},
      volume={3},
       pages={1417–1458},
}

\bib{de2022time}{article}{
      author={de~Le{\'o}n, M.},
      author={Gaset, J.},
      author={Gr{\`a}cia, X.},
      author={Mu{\~n}oz-Lecanda, M.C.},
      author={Rivas, X.},
       title={Time-dependent contact mechanics},
        date={2022},
     journal={Monatshefte f{\"u}r Mathematik},
       pages={1\ndash 35},
}

\bib{canaria}{article}{
      author={de~Le\'on, M.},
      author={Lainz, M.},
       title={A review on contact {H}amiltonian and {L}agrangian systems},
        date={2019},
     journal={Revista de la Real Academia de Ciencias Canaria},
      volume={XXXI},
       pages={1\ndash 46},
}

\bib{de2019contact}{article}{
      author={de~Le{\'o}n, M.},
      author={Lainz-Valc{\'a}zar, M.},
       title={Contact {H}amiltonian {S}ystems},
        date={2019},
     journal={Journal of Mathematical Physics},
      volume={60},
      number={10},
       pages={102902},
}

\bib{de2011methods}{book}{
      author={de~Le{\'o}n, M.},
      author={Rodrigues, P.R.},
       title={Methods of {D}ifferential {G}eometry in {A}nalytical
  {M}echanics},
   publisher={North-Holland, Amasterdam},
        date={1989},
}

\bib{de1993cosymplectic}{article}{
      author={de~Le{\'o}n, M.},
      author={Saralegi, M.},
       title={Cosymplectic reduction for singular momentum maps},
        date={1993},
     journal={Journal of Physics A: Mathematical and General},
      volume={26},
      number={19},
       pages={5033},
}

\bib{vaquero}{article}{
      author={de~León, M.},
      author={de~Diego, D.~Martín},
      author={Vaquero, M.},
       title={{Hamilton-{J}acobi theory, symmetries and coisotropic
  reduction}},
        date={2017},
     journal={J. Math. Pures Appl.},
      volume={107},
      number={5},
       pages={591\ndash 614},
}

\bib{de2022time2}{article}{
      author={de~León, M.},
      author={Lainz-Valcázar, M.},
      author={López-Gordón, A.},
      author={Rivas, X.},
       title={{H}amilton-{J}acobi theory and integrability for autonomous and
  non-autonomous contact systems},
        date={2023},
     journal={J. Geom. Phys.},
      volume={187},
       pages={Paper No. 104787, 22 pp},
}

\bib{miguel}{article}{
      author={de~León, M.},
      author={Lainz-Valcázar, M.},
      author={Muñoz-Lecanda, M.C.},
       title={Optimal control, contact dynamics and {H}erglotz variational
  problem},
        date={2023},
     journal={Journal of Nonlinear Science},
      volume={33 (1)},
       pages={9, 46 pp.},
}

\bib{deleon2003tulczyjews}{article}{
      author={de~León, M.},
      author={Martin~de Diego, D.},
      author={Santamaria-Merino, A.},
       title={Tulczyjew's triples and {L}agrangian submanifolds in classical
  field theories},
        date={2003},
     journal={Willy Sarlet and Frans Cantrijn , Applied Differential Geometry
  and Mechanics. Gent, Academia Press},
       pages={21\ndash 47},
}

\bib{silvia}{book}{
      author={de~León, M.},
      author={Salgado, M.},
      author={Vilariño, S.},
       title={Methods of differential geometry in classical field theories},
   publisher={World Scientific Publishing Co. Pte. Ltd., Hackensack, NJ, 2016,
  xiii+207 pp.},
        date={2014},
}

\bib{delucas2023cosymplectic}{misc}{
      author={de~Lucas, J.},
      author={Maskalaniec, A.},
      author={Zawora, B.~M.},
       title={Cosymplectic geometry, reductions, and energy-momentum methods
  with applications},
        date={2023},
}

\bib{Dirac_1950}{article}{
      author={Dirac, P.A.M.},
       title={Generalized {H}amiltonian dynamics},
        date={1950},
     journal={Canadian Journal of Mathematics},
      volume={2},
       pages={129–148},
}

\bib{ogulin}{article}{
      author={Esen, O},
      author={de~León, M.},
      author={Lainz-Valcázar, M.},
      author={Sardón, C.},
      author={Zając, M.},
       title={Reviewing the geometric {H}amilton–{J}acobi theory concerning
  {J}acobi and {L}eibniz identities},
        date={2022},
     journal={Journal of Physics A: Mathematical and Theoretical, 55 (40),
  403001},
}

\bib{ogul}{article}{
      author={Esen, O.},
      author={Gümral, H.},
       title={Tulczyjew's triplet for {L}ie groups {I}: Trivializations and
  reductions},
        date={2014},
     journal={J. Lie Theory},
      volume={24},
      number={2},
       pages={1115\ndash 1160},
}

\bib{ogul2}{article}{
      author={Esen, O.},
      author={Gümral, H.},
       title={Tulczyjew's triplet for {L}ie groups ii: Dynamics},
        date={2017},
     journal={J. Lie Theory 27, no. 2, 329–356},
}

\bib{ogul2021}{article}{
      author={Esen, O.},
      author={Lainz-Valcázar, M.},
      author={de~León, M.},
      author={Marrero, J.C.},
       title={Contact dynamics: {L}egendrian and {L}agrangian submanifolds},
        date={2021},
     journal={Mathematics},
      volume={9},
      number={21},
       pages={2704, 41pp},
}

\bib{garcia2022momentum}{article}{
      author={Garc{\'\i}a-Mauri{\~n}o, J.M.},
       title={{Momentum mapping and reduction in contact {H}amiltonian
  systems}},
        date={2022},
     journal={arXiv preprint arXiv:2208.10924},
}

\bib{godbillon1969geometrie}{book}{
      author={Godbillon, C.},
       title={G{\'e}om{\'e}trie diff{\'e}rentielle et m{\'e}canique
  analytique},
   publisher={Editions Hermann, Paris},
        date={1969},
}

\bib{got}{article}{
      author={Gotay, M.J.},
       title={On coisotropic imbeddings of presymplectic manifolds},
        date={1982},
     journal={Proc. Amer. Math. Soc.},
      volume={84},
      number={1},
       pages={111\ndash 114},
}

\bib{GotaySingularLagrangians1}{article}{
      author={Gotay, M.J.},
      author={Nester, J.M.},
       title={Presymplectic {L}agrangian systems. {I}. {T}he constraint
  algorithm and the equivalence theorem},
        date={1979},
        ISSN={0246-0211},
     journal={Ann. Inst. H. Poincar\'{e} Sect. A (N.S.)},
      volume={30},
      number={2},
       pages={129\ndash 142},
      review={\MR{535369}},
}

\bib{GotaySinguarLagrangians2}{article}{
      author={Gotay, M.J.},
      author={Nester, J.M.},
       title={Presymplectic {L}agrangian systems. {II}. {T}he second-order
  equation problem},
        date={1980},
        ISSN={0246-0211},
     journal={Ann. Inst. H. Poincar\'{e} Sect. A (N.S.)},
      volume={32},
      number={1},
       pages={1\ndash 13},
      review={\MR{574809}},
}

\bib{grabowska}{article}{
      author={Grabowska, K.},
       title={A {T}ulczyjew triple for classical fields},
        date={2012},
     journal={Journal of Physics A: Mathematical and Theoretical},
      volume={45},
      number={14},
       pages={145207},
         url={https://dx.doi.org/10.1088/1751-8113/45/14/145207},
}

\bib{jcelisa}{article}{
      author={Guzmán, E.},
      author={Marrero, J.C.},
       title={Time-dependent mechanics and {L}agrangian submanifolds of
  presymplectic and {P}oisson manifolds},
        date={2010},
     journal={J. Phys. A: Math. Theor.},
      volume={43},
      number={50},
       pages={505201 (23pp)},
}

\bib{beppe}{article}{
      author={Ibort, A.},
      author={de~León, M.},
      author={Marmo, G.},
       title={Reduction of {J}acobi manifolds},
        date={1997},
     journal={J. Phys. A},
      volume={30},
      number={8},
       pages={2783–2798},
}

\bib{regular}{article}{
      author={Ibort, A.},
      author={Marín-Solano, J.},
       title={Coisotropic regularization of singular lagrangians},
        date={1995},
     journal={J. Math. Phys.},
      volume={36},
      number={10},
       pages={5522–5539},
}

\bib{raul}{article}{
      author={Ibáñez, R.},
      author={de~León, M.},
      author={Marrero, J.C.},
      author={Martín~de Diego, D.},
       title={Co-isotropic and {L}egendre-{L}agrangian submanifolds and
  conformal {J}acobi morphisms},
        date={1997},
     journal={J. Phys. A},
      volume={30},
      number={15},
       pages={5427–5444},
}

\bib{kostant}{inproceedings}{
      author={Kostant, B.},
       title={Quantization and unitary representations},
        date={1970},
   booktitle={Lectures in modern analysis and applications iii},
      editor={Taam, C.~T.},
   publisher={Springer Berlin Heidelberg},
     address={Berlin, Heidelberg},
       pages={87\ndash 208},
}

\bib{Lainztesis}{thesis}{
      author={Lainz-Valc{\'a}zar, M.},
       title={Contact {H}amiltonian {S}ystems},
        type={Ph.D. Thesis},
        date={2022},
}

\bib{le2019deformations}{article}{
      author={L{\^e}, H.V.},
      author={Oh, Y.G.},
      author={Tortorella, A.G.},
      author={Vitagliano, L.},
       title={Deformations of coisotropic submanifolds in {J}acobi manifolds},
        date={2019},
      eprint={1410.8446},
}

\bib{libermann2012symplectic}{book}{
      author={Libermann, P.},
      author={Marle, C.M.},
       title={Symplectic {G}eometry and {A}nalytical {M}echanics},
   publisher={Springer Science \& Business Media},
        date={2012},
      volume={35},
}

\bib{marsden1974reduction}{article}{
      author={Marsden, J.},
      author={Weinstein, A},
       title={Reduction of symplectic manifolds with symmetries},
        date={1974},
     journal={Rep. Math. Phys. 5},
      number={1},
       pages={121\ndash 130},
}

\bib{marsdenbook}{book}{
      author={Marsden, J.E.},
       title={Lectures on {M}echanics},
   publisher={London Math. Soc. Lecture Note Ser., Cambridge University Press,
  Cambridge},
        date={1992},
      volume={174},
}

\bib{marsden1990reduction}{book}{
      author={Marsden, J.E.},
      author={Montgomery},
      author={R., T.S., Ratiu},
       title={Reduction, symmetry, and phases in mechanics},
   publisher={American Mathematical Soc.},
        date={1990},
      volume={436},
}

\bib{meyer}{incollection}{
      author={Meyer, K.~R.},
       title={Symmetries and integrals in mechanics},
        date={1973},
   booktitle={Dynamical systems ({P}roc. {S}ympos., {U}niv. {B}ahia,
  {S}alvador, 1971)},
   publisher={Academic Press, New York-London},
       pages={259\ndash 272},
      review={\MR{331427}},
}

\bib{MRUGALAcontinuouscontact}{article}{
      author={Mrugala, R.},
       title={Continuous contact transformations in thermodynamics},
        date={1993},
        ISSN={0034-4877},
     journal={Reports on Mathematical Physics},
      volume={33},
      number={1},
       pages={149\ndash 154},
  url={https://www.sciencedirect.com/science/article/pii/003448779390050O},
}

\bib{MRUGALAcontactstrcuture}{article}{
      author={Mrugala, R.},
      author={Nulton, J.D.},
      author={{Christian Schön}, J.},
      author={Salamon, P.},
       title={Contact structure in thermodynamic theory},
        date={1991},
        ISSN={0034-4877},
     journal={Reports on Mathematical Physics},
      volume={29},
      number={1},
       pages={109\ndash 121},
  url={https://www.sciencedirect.com/science/article/pii/003448779190017H},
}

\bib{roels}{article}{
      author={Roels, J.},
      author={Weinstein, A.},
       title={Functions whose {P}oisson brackets are constants},
        date={1971},
     journal={J. Mathematical Phys.},
      volume={12},
       pages={1482–1486},
}

\bib{sanzserna}{book}{
      author={Sanz-Serna, J.M.},
      author={Calvo, M.P.},
       title={Numerical {H}amiltonian problems},
   publisher={Chapman \& Hall, London},
        date={1994},
      volume={7},
}

\bib{simoes2020contact}{article}{
      author={Simoes, A.A.},
      author={de~León, M.},
      author={Valcázar, M.~Lainz},
      author={de~Diego, D.~Martín},
       title={Contact geometry for simple thermodynamical systems with
  friction},
        date={2020},
     journal={Proc. A. 476, no. 2241, 20200244, 16 pp.},
}

\bib{GeometrySomeThermo}{inproceedings}{
      author={Simoes, A.A.},
      author={Mar t{\'i'}~de Diego, D.},
      author={Laínz, M.},
      author={de~Le{\'o}n, M.},
       title={The geometry of some thermodynamic systems},
        date={2021},
   booktitle={Geometric structures of statistical physics, information
  geometry, and learning},
      editor={Barbaresco, Fr{\'e}d{\'e}ric},
      editor={Nielsen, Frank},
   publisher={Springer International Publishing},
     address={Cham},
       pages={247\ndash 275},
}

\bib{simoes2021ThermodynamicSystems}{incollection}{
      author={Simoes, A.A.},
      author={Martín~de Diego, D.},
      author={Laínz-Valc{\'a}zar, M.},
      author={de~Le{\'o}n, M.},
       title={{Geometric {S}tructures of {S}tatistical {P}hysics, {I}nformation
  {G}eometry, and {L}earning, Springer Proc. Math. Stat., 361}},
        date={2021},
      editor={Barbaresco, F.},
      editor={Nielsen, F.},
   publisher={Springer, Cham},
       pages={247\ndash 275},
}

\bib{Sniatycki1983ReductionAQ}{article}{
      author={Sniatycki, J.},
      author={Weinstein, A.},
       title={Reduction and quantization for singular momentum mappings},
        date={1983},
     journal={Letters in Mathematical Physics},
      volume={7},
       pages={155\ndash 161},
         url={https://api.semanticscholar.org/CorpusID:120255451},
}

\bib{SouriauDynamical}{book}{
      author={Souriau, J.M.},
       title={Structure of dynamical systems},
      series={Progress in Mathematics},
   publisher={Birkh\"{a}user Boston, Inc., Boston, MA},
        date={1997},
      volume={149},
        ISBN={0-8176-3695-1},
         url={https://doi.org/10.1007/978-1-4612-0281-3},
        note={A symplectic view of physics, Translated from the French by C. H.
  Cushman-de Vries, Translation edited and with a preface by R. H. Cushman and
  G. M. Tuynman},
      review={\MR{1461545}},
}

\bib{stefan}{article}{
      author={Stefan, P.},
       title={Accessible {S}ets, {O}rbits, and {F}oliations with
  {S}ingularities},
        date={1974},
     journal={Proc. London Math. Soc.},
      volume={s3-29},
       pages={699\ndash 713},
}

\bib{sussmann1973orbits}{article}{
      author={Sussmann, H.J.},
       title={Orbits of families of vector fields and integrability of
  distributions},
        date={1973},
     journal={Transactions of the American Mathematical Society},
      volume={180},
       pages={171\ndash 188},
}

\bib{Tortorella2017}{article}{
      author={Tortorella, A.G.},
       title={Rigidity of integral coisotropic submanifolds of contact
  manifolds},
        date={2017sep},
     journal={Letters in Mathematical Physics},
      volume={108},
      number={3},
       pages={883\ndash 896},
}

\bib{tulczyjew1976hamiltonienne}{article}{
      author={Tulczyjew, W.M.},
       title={{Les sous-variet{\'e}s {L}agrangiennes et la dynamique
  {H}amiltonienne}},
        date={1976},
     journal={C. R. Acad. Sci. Paris S{\`e}r. A-B},
      volume={283},
       pages={A15\ndash A18},
}

\bib{tulczyjew1976lagrangienne}{article}{
      author={Tulczyjew, W.M.},
       title={{Les sous-variet{\'e}s {L}agrangiennes et la dynamique
  {L}agrangienne}},
        date={1976},
     journal={C. R. Acad. Sci. Paris S{\`e}r. A-B},
      volume={283},
       pages={A65\ndash A78},
}

\bib{vaisman2012lectures}{book}{
      author={Vaisman, I.},
       title={{Lectures on the geometry of {P}oisson manifolds}},
   publisher={Birkh{\"a}user},
        date={2012},
      volume={118},
}

\bib{weinsteincreed}{article}{
      author={Weinstein, A.},
       title={Symplectic manifolds and their {L}agrangian submanifolds},
        date={1971},
     journal={Advances in Math.},
      volume={6},
       pages={329–346},
}

\bib{weinstein1977lectures}{book}{
      author={Weinstein, A.},
       title={Lectures on {S}ymplectic {M}anifolds},
   publisher={American Mathematical Soc.},
        date={1977},
      number={29},
}

\bib{willett}{article}{
      author={Willett, C.},
       title={Contact reduction},
        date={2002},
     journal={Trans. Am. Math. Soc.},
      volume={354},
       pages={4245\ndash 5260},
}

\end{biblist}
\end{bibdiv}

\end{document}